\documentclass[10pt]{article}
\usepackage{amsmath}
\usepackage{mathrsfs}
\usepackage{amssymb}
\usepackage{latexsym}
\usepackage{amsthm}
\usepackage{amsmath,amssymb}
\usepackage{bm}
\usepackage{graphicx}
\usepackage{ascmac}
\usepackage{enumerate}
\usepackage[margin=20truemm]{geometry}
\theoremstyle{plain}

\newtheorem{prop}{Proposition}[section]
\newtheorem{lemma}{Lemma}[section]
\newtheorem{theorem}{Theorem}[section]

\newtheorem{remark}{Remark}[section]

\numberwithin{equation}{section}

\newcommand{\Eb}{\mathbb{E}}

\newcommand{\vare}{\varepsilon}

\allowdisplaybreaks

\title{Maximum of the Gaussian interface model 
in random external fields}
\author{Hironobu Sakagawa
\thanks{
Department of Mathematics, 
Faculty of Science and Technology, Keio University, 
3-14-1 Hiyoshi, Kouhoku-ku, Yokohama 223-8522, JAPAN. 
{\it E-mail address\/}: sakagawa{\char'100}math.keio.ac.jp}}

\begin{document}

\date{}

\maketitle


\begin{abstract}
We consider the 
Gaussian interface model in the presence of 
random external fields, that is 
the finite volume (random) Gibbs measure 
on $\mathbb{R}^{\Lambda_N}$, 
$\Lambda_N=[-N, N]^d\cap \mathbb{Z}^d$ 
with Hamiltonian
$H_N(\phi)= 
\frac{1}{4d}\sum\limits_{x\sim y}(\phi(x)-\phi(y))^2
-\sum\limits_{x\in \Lambda_N}\eta(x)\phi(x)$ 
and $0$-boundary conditions. 
$\{\eta(x)\}_{x\in \mathbb{Z}^d}$ is 
a family of i.i.d. symmetric random variables. 
We study how the typical maximal height of a 
random interface is modified by the addition of 
quenched bulk disorder. 
We show that the asymptotic behavior of the maximum 
changes depending on the tail behavior of the 
random variable $\eta(x)$ when $d\geq 5$. 
In particular, we identify the leading order asymptotics 
of the maximum. 
\end{abstract}

\begin{center}{\small\bfseries Key words}
{\small Random interfaces, Gaussian free fields, 
disordered systems, entropic repulsion.} 
\end{center}

2020 Mathematics Subject Classification; 82B24, 82B44, 60K35. 


\section{Introduction}
\subsection{Model and result}

The study of the many 
equilibrium/non-equilibrium phenomena taking place at 
interfaces is an important and fascinating 
field in statistical mechanics. 
In particular, the probabilistic study of 
the effective interface model has been 
an active area of research for several decades. 
The main purpose of this paper is to study 
how the typical maximal height of a 
random interface is modified by the addition of 
quenched bulk disorder. 

First of all, we introduce our model. 
Let $d\geq 1$ and $\Lambda_{N}=[-N, N]^d\cap
\mathbb{Z}^d$.
For a configuration
$\phi=\{\phi(x)\}_{x\in \Lambda_N}\in \mathbb{R}^{\Lambda_N}$,
we consider the following massless Hamiltonian with quadratic 
interaction potential and $0$-boundary conditions: 
\begin{equation*}
H_N(\phi) =
\frac{1}{4d}\sum_{\substack{\{x, y\}\cap \Lambda_N\ne\phi \\ |x-y|=1}}
(\phi(x)-\phi(y))^2
\Big|_{\phi(x) \equiv 0 \text{ for every } x\in
\partial^+ \Lambda_N} ,
\end{equation*}
where
$\partial^+\Lambda_N\! =\{x\notin \Lambda_N; |x-y| = 1 \text{\! for some }
y\in \Lambda_N\}$ and 
we take the summation over nearest neighbor pairs 
$\{x,y\}\subset \mathbb{Z}^d\times\mathbb{Z}^d$ with $|x-y|=1$ 
which intersect $\Lambda_N$ and each pair counted only once. 
The corresponding finite volume Gibbs measure is defined by the following: 
\begin{equation}\label{00}
\mu_N(d\phi) = \frac{1}{Z_N}\exp\bigl\{
-H_N(\phi)\bigl\}
\prod\limits_{x\in \Lambda_N}d\phi(x), 
\end{equation}
where 
$d\phi(x)$ denotes Lebesgue measure on $\mathbb{R}$ and
$Z_N$ is a normalization factor. 
The configuration
$\phi$
is interpreted as an effective modelization of a random
phase separating interface
embedded in $d+1$-dimensional space.
The spin $\phi(x)$ at site $x\in \Lambda_{N}$
denotes the height of the interface. 

When the Hamiltonian is generally given by the form 
\begin{equation*}
\widetilde{H}_N(\phi) =
\sum_{\substack{\{x, y\}\cap \Lambda_N\ne\phi \\ |x-y|=1}}
V(\phi(x)-\phi(y))
\Big|_{\phi(x) \equiv 0 \text{ for every } x\in
\partial^+ \Lambda_N} ,
\end{equation*}
for the interaction potential $V:\mathbb{R}\to \mathbb{R}$, 
the model (\ref{00}) is called $\nabla\phi$ interface model or 
lattice massless field and its study 
has been quite active from both dynamical and statical aspects 
over the last few decades. 
In our case, by summation by parts we have 
\begin{align}\label{sbp}
H_N(\phi)
= \frac{1}{2}\left<\phi, (-\Delta_N)\phi\right>_{\Lambda_N}, 
\end{align}
where 
$\Delta_N$ is a discrete Laplacian on $\mathbb{Z}^d$ 
with Dirichlet boundary condition outside $\Lambda_N$, 
namely, 
\begin{equation*}
\Delta_N(x, y) = 
\begin{cases}
\frac{1}{2d} \ \ \text{ if } x, y\in \Lambda_N, |x-y|=1, \\
-{1} \ \ \text{ if } x, y \in \Lambda_N, x=y, \\
\, 0 \ \ \text{ otherwise. } 
\end{cases}
\end{equation*}
$\langle \ \cdot \ , \ \cdot \ \rangle_{\Lambda_N}$ 
denotes $l^2(\Lambda_N)$-scalar product. 
Therefore, $\mu_N$ coincides with the law of a centered Gaussian field 
on $\mathbb{R}^{\Lambda_N}$ whose covariance matrix is given by 
$(-\Delta_N)^{-1}$. 
This is often called lattice Gaussian free field. 
It is well-known that $(-\Delta_N)^{-1}$ has 
a random walk representation: 
\begin{equation*}
(-\Delta_N)^{-1}(x, y) = G_N(x, y):= 
{E}_x\Bigl[\sum\limits_{n=0}^\infty 
I(S_n =y, n<\tau_{\Lambda_N})\Bigr], 
\end{equation*}
where 
$\{S_n\}_{n\geq 0}$ is a simple random walk on $\mathbb{Z}^d$,  
${E}_x$ denotes expectation 
starting at $S_0=x\in \mathbb{Z}^d$ and 
$\tau_A= \inf\{ n\geq 0; S_n \notin A\}$ is the first exit time 
from $A\subset \mathbb{Z}^d$. 
By standard estimates on the Green function of 
the simple random walk (cf.\cite{L}), 
the following asymptotic behavior of the 
variance holds:  
\begin{equation}\label{var}
\mathrm{Var}_{\mu_N}(\phi(0))
= 
G_N(0, 0) \sim
\begin{cases}
g_1 N & \text{ if } d=1, \\ 
g_2 \log N & \text{ if } d=2, \\ 
G^*& \text{ if } d\geq 3, \\ 
\end{cases}
\end{equation}
for some $g_1, g_2>0$ 
and $G^* :=(-\Delta)^{-1}(0, 0)<\infty$ 
where 
$\Delta$ is a discrete Laplacian on $\mathbb{Z}^d$. 
$\sim$ means that the ratio of the both sides 
converges to $1$ in the limit $N\to\infty$. 
Under $\mu_N$, the interface is said to be delocalized when $d=1, 2$ 
because the variance diverges as $N\to\infty$. 
While, when $d\geq 3$ 
the interface is localized because the variance remains finite. 

One of the problems related to random interfaces 
is the study of the behavior of the field under the 
effect of various external potentials. 
For self-potentials 
$U_x:\mathbb{R}\to \mathbb{R}$, $x\in \Lambda_N$, 
the corresponding model is generally defined as follows: 
\begin{equation*}\label{square}
\mu_{N}^U (d\phi) = \frac{1}{Z_{N}^U}
\exp\bigl\{
-{H}_N (\phi) - \sum\limits_{x\in \Lambda_N} 
U_x(\phi(x)) \bigl\}
\prod\limits_{x\in {\Lambda_N}} d\phi(x), 
\end{equation*}
where 
${Z}_{N}^U$ is a normalization factor. 
Two most studied types of the potential are 
pinning: $U_x(r)= -bI(|r|\leq a)$, $a>0$, $b>0$ 
and hard wall: $U_x(r)=\infty \cdot I(r<0)$. 
See review articles \cite{F05} and \cite{V06} 
for the background and development of these problems. 
\cite{S09} and \cite{S21} studied the effect of 
weak repulsive potentials instead of the hard wall. 
Also, 
disordered pinning of the form 
$U_x(r)= (\beta \omega(x)+h)I(|r| \leq 1)$ 
has been studied in detail recently 
where $\beta, h\in \mathbb{R}$ and 
$\{\omega(x)\}_{x\in \mathbb{Z}^d}$ is a family 
of i.i.d. random variables. 
(cf.~\cite{CM13}, \cite{GL18}, \cite{GL22}, etc.)

In this paper, we consider 
self-potentials given by the form 
$U_x(r) = -\eta(x)r,  x\in \Lambda_N$ 
where $\eta = \{\eta(x)\}_{x \in \mathbb{Z}^d}$ is a 
family of independent random variables 
defined on some other probability space. 
This describes the situation that 
the space is filled by a (random) media changing at each site of 
$\Lambda_N$. 
Namely, chemical potentials $\eta$ 
represent the effect of the quenched bulk disorder. 
The corresponding (random) Gibbs measure is defined as follows: 
\begin{align}\label{gibbs}
\mu_N^{\eta}(d\phi) 
=\frac{1}{Z_N^{\eta}}
\exp\Bigl\{
-\frac{1}{4d}\sum_{\substack{\{x, y\}\cap \Lambda_N\ne\phi \\ |x-y|=1}}
(\phi(x)-\phi(y))^2
\Big|_{\phi(x) \equiv 0 \text{ for every } x\in
\partial^+ \Lambda_N} 
 + \sum\limits_{x\in \Lambda_N}\eta(x)\phi(x) \Bigr\}
\prod\limits_{x\in \Lambda_N}d\phi(x), 
\end{align}
where 
$Z_N^{\eta}$ is a normalization factor. 
We briefly recall several known results for this model. 
Recently, Dario, Harel and Peled \cite{DHP23} 
studied the fluctuation of the field. 
Under the condition that 
the interaction potential is symmetric and strictly convex, 
$\eta=\{\eta(x)\}_{x\in \mathbb{Z}^d}$ 
are independent and $\mathrm{Var}_{\mathbb{P}}(\eta(x))
=\sigma^2 \in (0, \infty)$ 
for every $x\in \mathbb{Z}^d$, 
they proved the following (cf. \cite[Theorem 2]{DHP23}): 
\begin{align}\label{dhp}
\mathrm{Var}_{\mathbb{P}}(E^{\mu_N^\eta}[\phi(0)]) 
\approx 
\begin{cases}
N^{4-d} & \text{ if } d=1, 2, 3, \\ 
\log N & \text{ if } d=4, \\ 
1 & \text{ if } d\geq 5, 
\end{cases}
\end{align}
where $\approx$ means that the ratio of the both sides 
stays uniformly positive and bounded in $N\geq 1$. 
Compared to (\ref{var}), 
the bulk disorder enhances the fluctuation of the field. 
Existence and uniqueness of random gradient Gibbs measures have 
been studied in \cite{CK12}, \cite{CK15}, \cite{vEK08}. 
\cite{D23} studied the convergence of the sequence of finite volume 
random Gibbs measures $\{\mu_N^\eta\}_{N\geq 1}$ to the
infinite volume translation-covariant state. 
See \cite{D23}, \cite{DHP23} and references theirin 
for the background of this model and other previously known results. 

The main objective of this paper is to analyze the asymptotic behavior of 
the maximum of the field in the presence of random external fields. 
We are interested in  how the typical maximal height of a 
random interface are modified by the addition of 
quenched bulk disorder. 
In relation to the study of log-correlated random fields, 
the study of the maximum of the two-dimensional 
discrete Gaussian free field is 
an active area of research (cf.~\cite{bib6} and references theirin). 
For the higher dimensional case, 
\cite{bib8} and \cite{bib9} studied the scaling limit. 
The typical maximum under $\mu_N$ behaves like $\sqrt{2dG^*\log N}$ 
when $d\geq 3$. 
This asymptotics is the same as that of 
the family of i.i.d. Gaussian random variables 
$\mathcal{N}(0, G^*)$ on $\mathbb{R}^{\Lambda_N}$. 
For the maximum of the $\nabla\phi$ model 
with non-Gaussian interaction potentials, 
we are only aware of the result by \cite{bib1} 
which identified the leading order of the maximum in two dimensions 
under the condition that the interaction potential is 
symmetric and strictly convex. 
We could not find any other result 
including the higher dimensional case. 

Now, we are in the position to state the main result of this paper. 
We assume that 
$\eta=\{\eta(x)\}_{x\in \mathbb{Z}^d}$ is a family of i.i.d. 
symmetric random variables which satisfies 
either of the following tail asymptotics:  
\begin{align*}
(A_\alpha) :
& 
\text{ there exist } \alpha \in (0, 2] \text{ and } c_\alpha \in (0, \infty) 
\text{ such that }
\lim\limits_{r\to \infty}\frac{1}{r^\alpha}\log 
\mathbb{P}(\eta(0) \geq r) 
= -c_\alpha, \\
(\widetilde{A}) :
& 
\lim\limits_{r\to \infty}\frac{1}{r^2}\log 
\mathbb{P}(\eta(0) \geq r) 
= -\infty. 
\end{align*}
Note that if $\eta(x)$ is a bounded random variable, 
then the case is included 
in the condition $(\widetilde{A})$. 
Throughout the paper below 
$\mathbb{P}$ denotes the probability measure over the random field 
$\eta$ and $\mathbb{E}$ denotes the corresponding expectation. 
We interpret the exponent $\alpha \wedge 2$ as $2$ 
under the condition $(\widetilde{A})$. 
\begin{theorem}\label{thm1}
Let $d\geq 5$ 
and $\eta=\{\eta(x)\}_{x\in \mathbb{Z}^d}$ be a family of i.i.d. 
symmetric random variables. 
Assume the condition $(A_\alpha)$ or $(\widetilde{A})$ 
and define $M^*$ as follows: 
\begin{align}\label{MM}
M^*= 
\begin{cases}
(\frac{d}{c_\alpha})^{\frac{1}{\alpha}} G^* 
& \text{ if }\ (A_\alpha)\ \alpha \in (0, 1], \\
(\frac{d}{c_\alpha})^{\frac{1}{\alpha}} 
(G^{*}_{(\alpha)})^{\frac{\alpha-1}{\alpha}} 
& \text{ if }\ (A_\alpha)\ \alpha \in (1, 2), \\
(2dG^*+
\frac{d}{c_2}{G}^{*}_{(2)})^{\frac{1}{2}} 
& \text{ if }\ (A_\alpha)\ \alpha=2, \\
({2dG^*})^{\frac{1}{2}}
& \text{ if }\ (\widetilde{A}), 
\end{cases}
\end{align}
where $G^*=(-\Delta)^{-1}(0, 0)$ and 
${G}^{*}_{(\alpha)} =\sum\limits_{x\in \mathbb{Z}^d}
 ((-\Delta)^{-1}(0, x))^{\frac{\alpha}{\alpha-1}}$. 
Then, for every $\vare, \vare' >0$ we have 
\begin{equation}\label{maxglb}
\lim\limits_{N\to\infty}
\mu_N^{\eta} \Bigl( \frac{1}{(\log N)^{\frac{1}{\alpha \wedge 2}}} \max_{x \in \Lambda_N} 
\phi(x) \geq M^*-\vare \Bigr) =1,\ 
\mathbb{P}\text{-a.s.},
\end{equation}
and
\begin{equation}\label{maxgub}
\lim_{N\to\infty}
\mathbb{P}\Bigl(
\mu_N^{\eta} \Bigl( \frac{1}{(\log N)^{\frac{1}{\alpha \wedge 2}}} \max_{x \in \Lambda_N} 
\phi(x) \leq M^* +\vare \Bigr) \geq 1-\vare' 
\Bigr)= 1. 
\end{equation}
Moreover, in the case $(A_\alpha)$ $\alpha=2$ and 
$2dG^* > \frac{G_{(2)}^*}{c_2}$, or $(\widetilde{A})$ 
we have 
\begin{equation}\label{maxgub2}
\lim\limits_{N\to\infty}
\mu_N^{\eta} \Bigl( \frac{1}{(\log N)^{\frac{1}{\alpha \wedge 2}}} \max_{x \in \Lambda_N} 
\phi(x) \leq M^*+\vare \Bigr) =1,\ 
\mathbb{P}\text{-a.s.}. 
\end{equation}

\end{theorem}

\smallskip

\noindent
By this result we can see that the asymptotics of the maximum of the field 
changes depending on the tail behavior of the disorder variable $\eta$ 
when $d\geq 5$. 
In particular, we identified its leading order. 
We give several remarks about the result. 
\begin{remark}
$(\ref{maxglb})$ and $(\ref{maxgub2})$ imply that 
$\frac{1}{(\log N)^{\frac{1}{\alpha \wedge 2}}} \max\limits_{x \in \Lambda_N}
\phi(x) $ 
converges to $M^*$ in $\mu_N^{\eta}$-probability 
in the quenched sense, namely, 
\begin{equation*}
\lim\limits_{N\to \infty} 
\mu_N^{\eta} \Bigl( \bigl| \frac{1}{(\log N)^{\frac{1}{\alpha \wedge 2}}} 
\max_{x \in \Lambda_N} \phi(x) -M^* \bigr| \geq \vare \Bigr) =0, \ 
\mathbb{P}\text{-a.s.}, 
\end{equation*}
for every $\vare>0$. 
On the other hand, 
$(\ref{maxglb})$ and $(\ref{maxgub})$ imply 
the convergence in the annealed sense, namely, 
\begin{equation*}
\lim\limits_{N\to \infty} {\mathbb{E}} \Bigl[
\mu_N^{\eta} \Bigl( \bigl| \frac{1}{(\log N)^{\frac{1}{\alpha \wedge 2}}} 
\max_{x \in \Lambda_N} \phi(x) -M^* \bigr| \geq \vare \Bigr)\Bigr] =0. 
\end{equation*}
The reason why we have this slightly weaker result under the condition 
$(A_\alpha)$ $\alpha \in (0, 2)$, or $\alpha=2$ and 
$2dG^* \leq \frac{G_{(2)}^*}{c_2}$
is the existence of large spikes of the (random) mean height of the field. 
See Remark \ref{rem1} below for details. 
\end{remark}

\begin{remark}
The choice of the quadratic interaction potential 
$V(r)=\frac{1}{4d}r^2$ is for the identity (\ref{sbp}) 
and notational simplicity. 
The similar result holds for the case with 
general quadratic potential 
$V(r) = \lambda r^2$, $\lambda>0$ 
by a simple change of variables. 
\end{remark}
\begin{remark}
The condition $d\geq 5$ corresponds to the existence of 
the (random) infinite volume Gibbs measure $\mu_\infty^\eta$. 
Actually, the same statement to Theorem \ref{thm1} holds 
even if we replace $\mu_{N}^\eta$ with $\mu_\infty^\eta$.
The detail is given in Appendix A.1.
\end{remark}
\begin{remark}
In the lower dimensional case $1\leq d\leq 4$, 
fluctuation of the (random) mean height 
is much larger than that of the original $\phi$-field. 
Namely, it holds that 
$\mathrm{Var}_{\mathbb{P}}\bigl(E^{\mu_N^\eta}[\phi(x)]\bigr) \gg 
\mathrm{Var}_{\mu_N^\eta}(\phi(x))$, $x\in \Lambda_N$. 
Also, the field of the mean height 
$\{E^{\mu_N^\eta}[\phi(x)]\}_{x\in \Lambda_N}$ 
under $\mathbb{P}$ has long range correlations. 
By these reasons, the quenched result such as 
Theorem \ref{thm1} does not hold. 
We discuss an example in Appendix A.2. 
\end{remark}

\subsection{On the result and strategy of the proof}
\noindent
Next, we explain the strategy of the proof of Theorem \ref{thm1}. 
Firstly we have the following simple observation: 
\begin{align*}
H_N(\phi) - \sum\limits_{x\in \Lambda_N}\eta(x)\phi(x) 
& = \frac{1}{2}\left<\phi, (-\Delta_N)\phi\right>_{\Lambda_N}
-\left<\phi, \eta \right>_{\Lambda_N} \\
& = 
\frac{1}{2}\left<(\phi-m_N^\eta), (-\Delta_N)(\phi-m_N^\eta) \right>_{\Lambda_N}- 
\frac{1}{2}\left<m_N^\eta, (-\Delta_N)m_N^\eta \right>_{\Lambda_N}, 
\end{align*}
where 
$m_N^\eta=\{m_N^\eta(x)\}_{x\in \Lambda_N}$ is given by 
\begin{align*}
m_N^\eta(x)= ((-\Delta_N)^{-1}\eta)(x) = 
\sum\limits_{y\in \Lambda_N}G_N(x, y)\eta(y),\ x\in \Lambda_N.
\end{align*} 
Therefore, the Gibbs measure (\ref{gibbs}) coincides with 
the law of the Gaussian field on $\mathbb{R}^{\Lambda_N}$ 
with (random) mean $m_N^\eta$ and covariance $(-\Delta_N)^{-1}$. 
The bulk disorder $\eta$ affects only the mean height of the field. 
Under the condition that $\{\eta(x)\}_{x\in \mathbb{Z}^d}$ 
is a family of i.i.d. symmetric random variables, we have 
$\mathbb{E}[m_N^\eta(x)]= 0$ and 
$\mathbb{E}[m_N^\eta(x) m_N^\eta(y)] = \sigma^2 (-\Delta_N)^{-2}(x, y)$ 
for every $x, y\in \Lambda_N$ where 
$\sigma^2 = \mathrm{Var}_{\mathbb{P}}(\eta(0))$ 
and $(-\Delta_N)^{-2}= (-\Delta_N)^{-1}(-\Delta_N)^{-1}$ denotes 
usual matrix product. 
Namely, 
$m_N^\eta$-field has long range correlations and 
its covariance structure is the same as the so-called $\Delta\phi$ model, 
or Gaussian membrane model 
(cf. \cite{CDH19}, \cite{K12}, \cite{S03}, etc.). 
On the other hand, 
different from the $\nabla\phi$ model or $\Delta\phi$ model, 
$m_N^\eta$-field 
does not have spatial Markov property (Gibbs-Markov property). 
For the asymptotics of $(-\Delta_N)^{-1}$, 
the following is well-known: 
\begin{align*}
(-\Delta_N)^{-1}(0, x) = 
\begin{cases}
N+1-|x| \ \ \text{ if } d=1, \\
C_2(\log N-\log (|x|+1) )+O(|x|^{-1}) \ \ \text{ if } d=2, \\
C_d(|x|^{2-d}-N^{2-d})+O(|x|^{1-d}) \ \ \text{ if } d\geq 3, 
\end{cases}
\end{align*}
for some constant $C_d >0$ 
(cf.~\cite[Proposition 1.5.9 and Proposition 1.6.7]{L}). 
In particular, for $d\geq 3$ we have 
\begin{align*}
(-\Delta_N)^{-2}(0, 0) = \sum\limits_{x\in \Lambda_N}
((-\Delta_N)^{-1}(0, x))^2 
\sim C\sum\limits_{r=1}^N r^{d-1}\bigl(\frac{1}{r^{d-2}}\bigr)^2 
= C\sum\limits_{r=1}^N r^{-d+3}, 
\end{align*}
as $N\to \infty$. 
By the similar computation for $d=1, 2$, we obtain 
\begin{align*}
\mathrm{Var}_{\mathbb{P}} (m_N^\eta(0)) 
= \sigma^2 (-\Delta_N)^{-2}(0, 0) 
\approx 
\begin{cases}
N^{4-d} & \text{ if } d=1, 2, 3, \\ 
\log N & \text{ if } d=4, \\ 
1 & \text{ if } d\geq 5.  
\end{cases}
\end{align*}
as $N\to \infty$. 
Hence $m_N^\eta$-field is localized only if $d\geq 5$ 
and in the lower dimensional case $1\leq d\leq 4$, 
fluctuation of the (random) mean height under $\mu_N^\eta$ 
dominates that of the original $\phi$-field. 
Actually this is the same estimate as (\ref{dhp}) 
which was obtained by \cite{DHP23} 
for the general interaction potential case. 

By this observation, we study the asymptotic behavior 
of the random field 
determined from the weighted sum of i.i.d. random variables 
whose weights are given by the Green function of the simple random walk. 
We give the precise deviation estimates 
depending on the tail assumption of the disorder variable $\eta$ 
(see Lemma \ref{a1} below). 
In particular we can prove the following: 
\begin{itemize}
\item
Under the condition $({A}_\alpha)$ $\alpha \in (0, 2)$, 
the fluctuation of 
$\max\limits_{x\in \Lambda_N}m_N^\eta(x)$ under $\mathbb{P}$ 
dominates that of $\max\limits_{x\in \Lambda_N}\phi(x)$ under $\mu_N$. 
In particular, the asymptotics of 
$\max\limits_{x\in \Lambda_N}\phi(x)$ under $\mu_N^{\eta}$ 
is determined from that of 
$\max\limits_{x\in \Lambda_N}m_N^\eta(x)$ under $\mathbb{P}$. 
\item
Under the condition $(\widetilde{A})$, 
the fluctuation of 
$\max\limits_{x\in \Lambda_N}\phi(x)$ under $\mu_N$ 
dominates that of 
$\max\limits_{x\in \Lambda_N}m_N^\eta(x)$ under $\mathbb{P}$ and 
the asymptotics of 
$\max\limits_{x\in \Lambda_N}\phi(x)$ under $\mu_N^{\eta}$ 
is determined from that of 
$\max\limits_{x\in \Lambda_N}\phi(x)$ under $\mu_N$. 
\end{itemize}

On the other hand, 
the case $(A_\alpha)$ $\alpha=2$ is much more delicate. 
In this case the tail asymptotics of $\phi$-field under 
$\mu_N$ and $m_N^\eta$-field under $\mathbb{P}$ 
are the same order. 
For the moment we assume that 
$\eta = \{\eta(x)\}_{x\in \mathbb{Z}^d}$ is a family of i.i.d. 
normal random variables $\mathcal{N}(0, \sigma^2)$. 
Under the product measure 
$(\eta, \phi) \sim \mathbb{P}\otimes \mu_N$, 
the field $\{\phi(x)+m_N^{\eta}(x)\}_{x\in \Lambda_N}$ 
is a centered Gaussian field on $\mathbb{R}^{\Lambda_N}$ 
whose covariance matrix is given by 
$(-\Delta_N)^{-1}+ \sigma^2(-\Delta_N)^{-2}$. 
In particular $\lim\limits_{N\to \infty} 
\mathrm{Var}_{\mathbb{P}\otimes \mu_N}
(\phi(0)+m_N^\eta(0))=
G^* + \sigma^2 G^*_{(2)}(<\infty \text{ if } d\geq 5)$. 
Now, recall the fact that 
the maximum of the correlated Gaussian field 
$\mathcal{N}(0, (\Delta_N)^{-1})$ on $\mathbb{R}^{\Lambda_N}$ 
shows the similar asymptotic behavior as 
that of the 
the family of i.i.d. Gaussian random variables 
$\mathcal{N}(0, G^*)$ on $\mathbb{R}^{\Lambda_N}$ when 
$d\geq 3$ where 
$G^*= \lim\limits_{N\to \infty}
\mathrm{Var}_{\mu_N}(\phi(0))$. 
By this fact we can guess that 
$\max\limits_{x\in \Lambda_N}\{\phi(x)+m_N^\eta(x)\}$ 
behaves as 
$\sqrt{2d(G^*+ \sigma^2 G^*_{(2)})\log N}$ 
and  actually this coincides with the result of 
Theorem \ref{thm1} 
including the precise constant (\ref{MM}). 
Of course this is a simple observation for the 
annealed Gaussian measure and 
this argument does not directly lead the result for the 
quenched case or general disorder $\eta$ 
which satisfies the condition $(A_\alpha)$ $\alpha=2$. 

The actual strategy of the proof of Theorem \ref{thm1} is as follows. 
For every $a\in \mathbb{R}$ we have 
\begin{align*}
\mu_N^{\eta}\bigl( 
\max\limits_{x\in \Lambda_N} \phi(x) \leq a \bigr) 
& = 
\mu_N\bigl( \max\limits_{x\in \Lambda_N} 
\{\phi(x) + m_N^\eta(x)\} \leq a \bigr)\\
& = 
\mu_N\bigl( \phi(x) \geq m_N^\eta(x)-a \text{ for every } 
x\in \Lambda_N \bigr). 
\end{align*}
Hence, the behavior of the maximum under the measure $\mu_N^{\eta}$ 
closely related to 
the probability that 
the Gaussian free field stays above a random wall. 
This problem has been studied by \cite{bib3} 
for the i.i.d. random wall case and by \cite{bib4} 
for the case that the wall itself is also the Gaussian free field. 
In our case, though the field 
$\{m_N^\eta(x); x \in \Lambda_N\}$ has long range correlations and 
does not have the spatial Markov property, 
we can show that the field 
$\{m_N^\eta(x); x \in \Lambda_N \cap 4L\mathbb{Z}^d\}$ 
can be treated as an independent random field for 
fixed $L$ large enough. 
Then, for the lower bound of the maximum, 
we can proceed the quenched-annealed comparison and 
the conditioning argument which have 
been developed in the study of the problem of entropic repulsion 
(cf.~\cite{bib3}, \cite{bib4}, \cite{S09}, etc.). 
The upper bound of the maximum is given 
by an FKG argument combining with the precise estimate 
on the number of large spikes for the $m_N^\eta$-field. 
In all the arguments the precise deviation estimates 
of the mean height play the crucial role. 


\subsection{Discussion}
Here we discuss some further research directions. 

\medskip

\noindent
{\em The non-Gaussian case}: 
The generalization of the interaction potential 
would be natural and extremely interesting. 
Let $\widetilde{\mu}_{N}^\eta$ be the finite volume 
(random) Gibbs measure on $\mathbb{R}^{\Lambda_N}$ 
corresponding to the Hamiltonian: 
\begin{align}\label{ham}
\widetilde{H}_N^\eta(\phi) =
\sum_{\substack{\{x, y\}\cap \Lambda_N\ne\phi \\ |x-y|=1}}
V(\phi(x)-\phi(y))
\Big|_{\phi(x) \equiv 0 \text{ for every } x\in
\partial^+ \Lambda_N}
- \sum\limits_{x\in \Lambda_N}\eta(x)\phi(x). 
\end{align}
We expect that under the assumption that the interaction 
potential $V$ is symmetric and strictly convex, 
the maximum depends on the tail asymptotics 
of the disorder $\eta$ and 
the similar behavior to Theorem \ref{thm1} holds 
(at least about the order of the maximum). 
So we consider the same storyline of 
the proof for the Gaussian case. 

Set $\widetilde{m}_N^\eta(x):= 
E^{\widetilde{\mu}_N^\eta}[\phi(x)]$, 
$x\in \Lambda_N$ and we decompose the $\phi$-field as 
$\phi(x)= 
(\phi(x)-\widetilde{m}_N^\eta(x)) + \widetilde{m}_N^\eta(x)$. 
By Brascamp-Lieb inequality (cf. \cite[Section 4.2]{F05}) 
the fluctuation of the field 
$\{\phi(x)-\widetilde{m}_N^\eta(x)\}_{x\in \Lambda_N}$ 
is comparable to that of the discrete Gaussian free field. 
For the mean height $\widetilde{m}_N^\eta(x)$ we have 
a random walk representation recently studied by 
\cite{DHP23} and 
$\widetilde{m}_N^\eta(x)$ can be represented as 
$\widetilde{m}_N^\eta(x)=\sum\limits_{y\in \Lambda_N}
G_N^\eta(x, y) \eta(y)$, $x\in \Lambda_N$ 
where 
$G_N^\eta(x, y)$ corresponds to the 
Green function of a random walk 
in dynamic random environment 
whose environment is determined from the 
Langevin equation associated with the Hamiltonian (\ref{ham}). 
\cite{DHP23} also proved the Nash-Aronson type estimate 
for the heat kernel of the random walk. 
Then, by using their results we can obtain the following quenched estimate: 
when $d\geq 3$, there exist $C_1, C_2>0$ such that 
$$\displaystyle 
\frac{C_1}{|x-y|^{d-2}\vee 1} 
\leq G_N^\eta(x, y) 
\leq 
\frac{C_2}{|x-y|^{d-2}\vee 1}, \ x, y\in \Lambda_N. 
$$
(For the lower bound we need some conditions for $x, y\in \Lambda_N$; 
see \cite[Section 5.1]{DHP23} for the precise statement.) 

Therefore, similarly to the Gaussian case, 
the mean height can be represented as a 
weighted sum of 
$\{\eta(x)\}_{x\in \Lambda_N}$ and 
its weights $\{G_N^\eta(x, y)\}_{x, y\in \Lambda_N}$ 
are comparable to the Green function of the simple random walk. 
However, the main difficulty is that 
the weights depend on $\{\eta(x)\}_{x\in \Lambda_N}$ and 
we do not know the dependence explicitly. 
Hence we cannot carry out the argument of the proof of 
Lemma \ref{a1} below for the non-Gaussian case and 
we don't have the control of the fluctuation of 
the field $\{\widetilde{m}_N^\eta(x)\}_{x\in \Lambda_N}$. 
This is the reason why we have not been able to prove the corresponding result 
for the non-Gaussian case. 

\medskip


\noindent 
{\em Scaling limit of the maximum}: 
For the case without disorder, 
\cite{bib8} and \cite{bib9} proved that 
the field belongs to the maximal domain 
of attraction of Gumbel 
distribution when $d\geq 3$, 
namely the following holds 
for an appropriate choice of $a_N$ and $b_N$. 
\begin{align}\label{gumb}
\lim\limits_{N\to \infty} 
\mu_N \bigl( 
\frac{1}{a_N}(\max\limits_{x\in \Lambda_N}\phi(x) -b_N) 
\leq z 
\bigr)
= \exp\{-e^{-z}\},\ z\in \mathbb{R}. 
\end{align}
Though the corresponding problem may be considered 
for the case with disorder, 
the proofs of \cite{bib8} and \cite{bib9} do not work well 
under the quenched measure $\mu_N^\eta$ because they 
heavily rely on the precise Gaussian computation 
such as Mills ratio or the estimate on the joint probability 
of two random variables and 
we cannot carry out such computation 
without the detailed information about $m_N^\eta(x)$. 
The convergence of the form (\ref{gumb}) essentially requires 
the identification of the second leading term of 
$\max\limits_{x\in \Lambda_N}\phi(x)$ 
and it seems to be quite difficult under $\mu_N^\eta$. 
Generally speaking, we don't know whether 
(\ref{gumb}) 
holds or not for a Gaussian field on $\mathbb{R}^{\Lambda_N}$ 
with spatially inhomogeneous mean which depends on the 
system size $N$. 

%
%


On the other hand, the convergence of the following 
annealed form might be possible: 
\begin{align*}
\mathbb{E}\Bigl[
\mu_N^\eta\bigl( \frac{1}{a_N}
({\max\limits_{x\in \Lambda_N}\phi(x) -b_N}) 
\leq z \bigr)\Bigl] 
& = 
\mathbb{P}\otimes 
\mu_N\Bigl( 
\frac{1}{a_N}({\max\limits_{x\in \Lambda_N}
\{\phi(x)+m_N^\eta(x)\} -b_N}) 
\leq z \Bigr). 
\end{align*}
If 
$\eta = \{\eta(x)\}_{x\in \mathbb{Z}^d}$ is a family of i.i.d. 
normal random variables $\mathcal{N}(0, \sigma^2)$, 
then the field $\{\phi(x)+m_N^{\eta}(x)\}_{x\in \Lambda_N}$ 
is a centered Gaussian field on $\mathbb{R}^{\Lambda_N}$ 
whose covariance matrix is given by 
$(-\Delta_N)^{-1}+ \sigma^2(-\Delta_N)^{-2}$ 
under the product measure 
$(\eta, \phi) \sim \mathbb{P}\otimes \mu_N$ and 
we can apply the result of \cite{bib8} when $d\geq 5$. 
The generalization of this problem 
leads the following mathematical problem: 
let $\phi=\{\phi(x)\}_{x\in \Lambda_N}$ be the 
discrete Gaussian free field on $\mathbb{R}^{\Lambda_N}$ and 
$\psi=\{\psi(x)\}_{x\in \Lambda_N}$ be a random field on 
$\mathbb{R}^{\Lambda_N}$ independent of $\phi$. 
Then, 
identify $\max\limits_{x\in \Lambda_N}\{\phi(x)+\psi(x)\}$ 
(and its scaling limit if possible).


\bigskip

In the rest of the paper,
we provide the proof of Theorems \ref{thm1} 
in Sections 2 and 3. 
In Section 2 we give deviation estimates on 
weighted sums of i.i.d. random variables 
which are needed for the study of the asymptotic behavior of 
the (random) mean height. 
The main part of the proof of Theorem \ref{thm1} 
is given in Section 3. 
In Appendix, we discuss the infinite volume case and the 
lower dimensional case. We also introduce the result for 
the problem of entropic repulsion. 
Throughout the paper $C$, $C'$, $C''$ represent positive constants 
that do not depend on the size of the system $N$, 
but may depend on other parameters. 
These constants in various 
estimates may change from place to place in the paper. 

\section{Deviations of weighted sums of i.i.d. random variables }
In this section we give deviation estimates 
of weighted sums of i.i.d. random variables 
whose weights are given by the Green function of the simple random walk. 
Let $G_N: \Lambda_N \times \Lambda_N \to \mathbb{R}$ be 
the Green function of the simple random walk with Dirichlet 
boundary conditions outside $\Lambda_N$. 
When $d\geq 3$, 
$G_N(x, y) \uparrow G(x, y):= (-\Delta)^{-1}(x, y)<\infty$ 
as $N\to \infty$ for every $x, y\in \mathbb{Z}^d$. 
In particular, we have 
$G_N(x, y) \leq G(0, y-x)\leq G(0, 0)=:G^*$ 
for every $x, y \in \Lambda_N$. 
We also know the following asymptotic behavior for $G(x, y)$. 

\smallskip

\begin{prop}\cite[Theorem 1.5.4]{L}\label{gfunc}
Let $d\geq 3$. 
As $|x-y| \to \infty$, 
$G(x, y)= G(0, y-x) \sim \frac{a_d}{|x-y|^{d-2}}$, where 
$a_d = \frac{d}{2}\Gamma(\frac{d}{2}-1)\pi^{-\frac{d}{2}}$. 
\end{prop}
For 
$\{\eta(x)\}_{x\in \mathbb{Z}^d} 
\in \mathbb{R}^{\mathbb{Z}^d}$, 
we define a weighted sum 
$S_A(x)=
\sum\limits_{y\in A}G_N(x, y)\eta(y)$, 
$A\subset \Lambda_N$, 
$x\in \mathbb{Z}^d$. 
$\Lambda_L(x)=x+\Lambda_L$ denotes a box centered at 
$x \in \mathbb{Z}^d$ and side length $2L+1$. 
In the following 
we denote the event $\bigcup\limits_{N=1}^\infty 
\bigcap\limits_{k=N}^\infty \mathcal{A}_k$ as 
$\{ \mathcal{A}_N, \text{ for all large } N \}$ 
for a sequence of events $\{\mathcal{A}_N\}$. 

\smallskip

\begin{lemma}\label{a1}
Let $d\geq 5$ and 
$\eta=\{\eta(x)\}_{x\in \mathbb{Z}^d}$ be a family of i.i.d. 
symmetric random variables. 
\begin{enumerate}[$(i)$]
\item
Assume that $\eta$ satisfies the condition $(A_\alpha)$. 
Then, for every $\delta>0$ there exists 
$L_0=L_0(\delta)\in \mathbb{N}$ such that the following holds for 
every $L \geq L_0$: 
\begin{equation*}
\mathbb{P}\bigl( 
|S_{\Lambda_N\cap (\Lambda_{L}(x))^c}(x)| \leq 
(\delta \log N)^{\frac{1}{\alpha}} \text{ for every } 
x \in \Lambda_{N}, \text{ for all large } N 
\bigr)=1.
\end{equation*}
\item
Let $L \in \mathbb{N}$ be fixed. 
Assume that $\eta$ satisfies the condition $(A_\alpha)$. 
Then, 
the following holds for every $x\in \Lambda_{N-L}$, 
$\vare>0$, $b>0$ 
and $N\in \mathbb{N}$ large enough: 
\begin{equation*}
\mathbb{P}\bigl( 
|S_{\Lambda_L(x)}(x)| \geq (b \log N)^{\frac{1}{\alpha}} 
\bigr) 
\begin{cases}
\leq N^{-(K - \vare) b}, \\
\geq N^{-(K + \vare) b}. 
\end{cases}
\end{equation*}
where $K$ is given by 
\begin{equation}\label{constK}
K = 
\begin{cases}
\frac{c_\alpha}{(G^*)^\alpha} &
\text{ if } \alpha\in (0, 1], \\
\frac{c_\alpha}{(G_{L, (\alpha)}^*)^{\alpha-1}} & 
\text{ if } \alpha\in (1, 2], 
\end{cases}
\end{equation}
$G^*=(-\Delta)^{-1}(0, 0)$ and 
$G_{L, (\alpha)}^* = \sum\limits_{x\in \Lambda_L} 
((-\Delta)^{-1}(0, x))^{\overline{\alpha}}$, 
$\overline{\alpha}=\frac{\alpha}{\alpha-1}$. 
\item
Assume that $\eta$ satisfies the condition $(\widetilde{A})$. 
Then, for every $\delta>0$ we have 
\begin{equation*}
\mathbb{P}\bigl( 
|S_{\Lambda_N}(x)| \leq (\delta 
\log N)^{\frac{1}{2}} \text{ for every } 
x \in \Lambda_{N}, \text{ for all large } N 
\bigr)=1.
\end{equation*}
\end{enumerate}
\end{lemma}
\noindent 
The proof of this lemma 
changes depending on the parameter $\alpha$. 
In the following, we set 
$b_N:= (b\log N)^{\frac{1}{\alpha \wedge 2}}$ 
for $b>0$. 
\begin{proof}[Proof of Lemma \ref{a1} for the case 
$(A_\alpha)$\ $0< \alpha <1$]
In this case $\eta$ does not have the exponential integrability. 
For the upper bound of the tail estimate we consider 
truncation and the exponential 
Chebyshev-type argument which was used in \cite{GHK06}. 
Let $\vare>0$ be small enough and take 
$A\subset \Lambda_N$ and $\lambda_N>0$. 
These are to be specified later on. 
For $x\in \Lambda_{N}$ and $b>0$, we have 
\begin{align*}
\mathbb{P}(S_A(x) \geq b_N) 
& \leq 
\mathbb{P}\bigl( 
\max_{y\in A}\{G_N(x, y)\eta(y)\} \leq b_N, S_A(x) \geq b_N
\bigr) + 
\mathbb{P}\bigl( 
\max_{y\in A}\{G_N(x, y)\eta(y)\} \geq b_N
\bigr) \\
& =: J_1(A) +J_2(A). 
\end{align*}

For the first term, 
\begin{align*}
J_1(A) 
& 
\leq 
e^{-\lambda_Nb_N} \mathbb{E}\Bigl[
\exp\bigl\{\lambda_N \sum_{y\in A} G_N(x, y)\eta(y) \bigr\}
I\bigl(
\max_{y\in A}\{G_N(x, y)\eta(y)\} \leq b_N
\bigr)
\Bigr]\\
& = 
e^{-\lambda_Nb_N} \prod\limits_{y\in A} 
\mathbb{E}\Bigl[
\exp\bigl\{\lambda_N G_N(x, y)\eta(y) \bigr\}
I\bigl( G_N(x, y)\eta(y) \leq b_N
\bigr)\Bigr]\\
& = 
e^{-\lambda_Nb_N} \prod\limits_{y\in A} 
\bigl\{ J_1^1 (y) + J_1^2(y) 
\bigr\},
\end{align*}
where 
\begin{align*}
J_1^1 (y) & = 
\mathbb{E}\Bigl[
\exp\bigl\{\lambda_N G_N(x, y)\eta(y) \bigr\}
I\bigl( \lambda_N G_N(x, y)\eta(y) \leq 1 \bigr)\Bigr], \\
J_1^2 (y) & = 
\mathbb{E}\Bigl[
\exp\bigl\{\lambda_N G_N(x, y)\eta(y) \bigr\}
I\bigl(1<\lambda_N G_N(x, y)\eta(y) \leq \lambda_N b_N \bigr)\Bigr]. 
\end{align*}
Since $e^t \leq 1 +t +t^2$ for every $t \leq 1$, 
we have 
\begin{align*}
J_1^1 (y) & \leq 
\mathbb{E}\Bigl[
\bigl\{ 1+ \lambda_N G_N(x, y)\eta(y) + 
\lambda_N^2G_N(x, y)^2 \eta(y)^2 \bigr\}
I\bigl( \lambda_N G_N(x, y)\eta(y) \leq 1 \bigr)\Bigr] \\
& \leq 
1 + \lambda_N^2 G_N(x, y)^2 \sigma^2 \\
& \leq 
\exp\bigl\{\lambda_N^2 G_N(x, y)^2 \sigma^2 \bigr\}, 
\end{align*}
where $\sigma^2= \mathrm{Var}_{\mathbb{P}}(\eta(0))$ and 
we used the assumption that the law of $\eta$ 
is symmetric for the second inequality. 
For $J_1^2(y)$, we use the following basic fact: 
for a random variable $X\geq 0$ and 
$0\leq a <b$, it holds that 
\begin{align*}
E[X I(a \leq X \leq b)] 
& = 
aP(X \geq a) -bP(X>b) + 
\int_{a}^{b} P(X\geq t) dt 
\\
& \leq 
aP(X \geq a)  + 
\int_{a}^{b} P(X\geq t) dt. 
\end{align*}
By this inequality and the tail assumption on $\eta$, we have 
\begin{align*}
J_1^2 (y) 
& \leq 
\int_{\frac{1}{\lambda_N}}^{b_N} \lambda_N e^{\lambda_N t} 
\mathbb{P} \bigl( G_N(x, y)\eta(y)\geq t \bigr) dt 
+ e \mathbb{P} \bigl( G_N(x, y)\eta(y) \geq \frac{1}{\lambda_N} 
\bigr) \\
& \leq 
\int_{\frac{1}{\lambda_N}}^{b_N} \lambda_N e^{\lambda_N t} 
e^{-(1-\vare)c_\alpha 
(\frac{t}{G_N(x, y)})^{\alpha}}dt 
+ e 
e^{-(1-\vare)c_\alpha 
(\frac{1}{\lambda_N G_N(x, y)})^{\alpha}}, 
\end{align*}
for every $N$ large enough 
where we take $\lambda_N$ so that 
$\lambda_N \to 0$ as $N\to \infty$.

\smallskip


Now we consider the case 
$A=\Lambda_N\cap (\Lambda_L(x))^c$. 
By Proposition \ref{gfunc} 
there exists a constant $G_+>0$ such that 
$G_N(x, y) \leq \frac{G_+}{1\vee |x-y|^{d-2}}$ 
for every $x, y \in \Lambda_N$. 
We take $\lambda_N$ as 
$\lambda_N= (1-2\vare) b_N^{\alpha-1} 
c_\alpha (\frac{L^{d-2}}{G_+})^\alpha$. 
Then, since $G_N(x, y) \leq \frac{G_+}{L^{d-2}}$ for 
every $y\in \Lambda_L(x)^c$, 
it is easy to see that 
\begin{equation*}
{\lambda_N t} {-(1-\vare)c_\alpha 
(\frac{t}{G_N(x, y)})^{\alpha}}
\leq -\vare c_\alpha (\frac{t}{G_N(x, y)})^{\alpha}, 
\end{equation*} 
for every $t \in [\frac{1}{\lambda_N}, b_N]$. 
This yields that 
\begin{align*}
J_1^2 (y) & \leq 
\lambda_N 
\int_{\frac{1}{\lambda_N}}^{b_N} 
e^{-\vare c_\alpha 
(\frac{t}{G_N(x, y)})^{\alpha}}dt 
+ e 
e^{-(1-\vare )c_\alpha 
(\frac{1}{\lambda_N G_N(x, y)})^{\alpha}} \\
& \leq 
\frac{\lambda_N G_N(x, y)}{\alpha} 
\int_{({\lambda_N G_N(x, y)})^{-\alpha}}^{\infty} 
e^{-\vare c_\alpha s}s^{\frac{1}{\alpha}-1}ds 
+ e 
e^{-(1-\vare )c_\alpha 
(\frac{1}{\lambda_N G_N(x, y)})^{\alpha}} \\
& \leq 
e^{-\frac{1}{2}\vare c_\alpha 
({\lambda_N G_N(x, y)})^{-\alpha}}, 
\end{align*}
for every $N$ large enough. 
In particular, $J_1^2(y)$ is negligible compared to 
$J_1^1(y)$. As a result, 
\begin{align*}
\prod_{y \in \Lambda_N\cap (\Lambda_L(x))^c}
\bigl\{J_1^1(y)+J_1^2(y)\bigr\} 
& \leq 
\prod_{y \in \Lambda_N\cap (\Lambda_L(x))^c}
\bigl\{
e^{\lambda_N^2G_N(x, y)^2\sigma^2}+ 
e^{-\frac{1}{2}\vare c_\alpha 
({\lambda_N G_N(x, y)})^{-\alpha}}
\bigr\} \\
& \leq 
\exp\Bigl\{(1+\vare) \lambda_N^2 \sigma^2 
\sum\limits_{y \in \Lambda_N\cap (\Lambda_L(x))^c} 
G_N(x, y)^2 \Bigr\} \\
& \leq 2, 
\end{align*}
where the last inequality follows from 
the fact that $\lambda_N \to 0$ as $N\to \infty$ and 
\begin{equation*}\label{G2}
\sum\limits_{y\in \Lambda_N\cap (\Lambda_L(x))^c} G_N(x, y)^2 
\leq C \sum\limits_{r\geq 1}r^{d-1}
(\frac{1}{r^{d-2}})^2 
\leq C\sum\limits_{r\geq 1}r^{-d+3}
<\infty, 
\end{equation*}
if $d\geq 5$. 
Hence, we obtain
\begin{align*}
J_1(\Lambda_N\cap (\Lambda_L(x))^c) 
& \leq 2e^{-\lambda_N b_N} 
\leq N^{-CL^{(d-2)\alpha}b}, 
\end{align*}
for every $N$ large enough 
where $C>0$ is a constant independent of $N$ and $L$. 

In the case $A=\Lambda_L(x)$, 
we take $\lambda_N$ as 
$\lambda_N= (1-2\vare) b_N^{\alpha-1} 
\frac{c_\alpha}{(G^*)^\alpha}$. 
By using $\max\limits_{y \in \Lambda_L(x)}
G_N(x, y) \leq {G^*}$, we can proceed the same argument 
as above and obtain 
\begin{align*}
J_1(\Lambda_L(x)) 
& \leq 2e^{-\lambda_N b_N} 
\leq N^{-(1-3\vare)\frac{c_\alpha}{(G^*)^\alpha} b}, 
\end{align*}
for every $N$ large enough. \\

For the estimate on $J_2(A)$, we use 
the estimates 
$\max\limits_{y \in \Lambda_N\cap (\Lambda_L(x))^c}
G_N(x, y) \leq \frac{G_+}{L^{d-2}}$ and 
$\max\limits_{y \in \Lambda_L(x)}
G_N(x, y) \leq {G^*}$ again. 
The union bound yields that 
\begin{align*}
J_2(\Lambda_N\cap (\Lambda_L(x))^c) 
& \leq 
\mathbb{P}\bigl(\max\limits_{y \in \Lambda_N\cap (\Lambda_L(x))^c}
\eta (y) \geq \frac{L^{d-2}}{G_+}b_N \bigr) \\
& \leq 
| \Lambda_N\cap (\Lambda_L(x))^c |\, 
\mathbb{P}\bigl(\eta (0) \geq \frac{L^{d-2}}{G_+}b_N \bigr) 
\leq 
N^{d-C'L^{(d-2)\alpha}b}, 
\end{align*}
and
\begin{align*}
J_2(\Lambda_L(x)) 
& \leq 
\mathbb{P}\bigl(\max\limits_{y \in \Lambda_L(x)}
\eta (y) \geq \frac{1}{G^*}b_N \bigr) \\
& \leq 
| \Lambda_L(x) |\, 
\mathbb{P}\bigl(\eta (0) \geq \frac{1}{G^*}b_N \bigr) 
\leq 
N^{-(1-\vare)\frac{c_\alpha}{(G^*)^{\alpha}}b}, 
\end{align*}
for every $N$ large enough. 

By combining these estimates, 
we obtain the upper bound of 
$
\mathbb{P}\bigl( 
S_{\Lambda_{L}(x)}(x) 
\geq (b \log N)^{\frac{1}{\alpha}}\bigr)$. 
Also, 
since the law of $S_A(x)$ is symmetric by the assumption on $\eta$, 
for every $\delta>0$ we can take $L_0=L_0(\delta)\in \mathbb{N}$ 
large enough so that
\begin{align*}
\mathbb{P}\bigl( 
|S_{\Lambda_N\cap (\Lambda_{L_0}(x))^c}(x)| 
\geq (\delta \log N)^{\frac{1}{\alpha}} 
\text{ for some } x \in \Lambda_{N}\bigr) 
& \leq 
2|\Lambda_{N}| \max_{x\in \Lambda_{N}} 
\mathbb{P}\bigl( 
S_{\Lambda_N\cap (\Lambda_{L_0}(x))^c}(x) \geq 
(\delta \log N)^{\frac{1}{\alpha}} \bigr) \\
& \leq 
C''N^d\bigl\{
N^{-CL^{(d-2)\alpha} \delta }+
N^{d-C'L^{(d-2)\alpha} \delta }\bigr\}\\
& 
\leq 
N^{-\gamma},
\end{align*}
for some $\gamma>1$ and every $N$ large enough. 
Therefore, we obtain (i) by 
Borel-Cantelli's lemma. 

For the lower bound of the tail estimate, 
we have 
\begin{align*}
\mathbb{P}(S_{\Lambda_L(x)}(x) \geq b_N) 
& \geq 
\mathbb{P}(G_N(x, x)\eta(x) \geq (1+\vare)b_N) 
\mathbb{P}(S_{\Lambda_L(x)\setminus \{x\}}(x) \geq 
-\vare b_N) \\
& = 
\mathbb{P}(G_N(x, x)\eta(x) \geq (1+\vare)b_N) 
\Bigl\{ 1- 
\mathbb{P}(S_{\Lambda_L(x)\setminus \{x\}}(x) < 
-\vare b_N) \Bigr\} \\
& \geq 
\mathbb{P}(G_N(x, x)\eta(x) \geq (1+\vare)b_N) 
\Bigl\{ 1- \frac{1}{\vare^2 b_N^2}
\mathbb{E}\bigl[ (S_{\Lambda_L(x)\setminus \{x\}}(x))^2\bigr] \Bigr\}. 
\end{align*}
By the assumption on $\eta$ and the fact that 
$G_N(x, x) \to G^*$ as $N\to \infty$, 
\begin{align*}
\mathbb{P}(G_N(x, x)\eta(x) \geq (1+\vare)b_N) 
& \geq 
\exp\Bigl\{
-(1+2\vare) c_\alpha \frac{b\log N}{(G^*-\vare)^\alpha} 
\Bigr\}, 
\end{align*}
for every $N$ large enough. We also have 
$$
\mathbb{E}\bigl[ (S_{\Lambda_L(x)\setminus \{x\}}(x))^2\bigr] 
 = 
\sigma^2 \sum\limits_{y\in \Lambda_L(x)\setminus \{x\}}
G_N(x, y)^2  \leq C < \infty.$$ 
As a result, we obtain the lower bound 
\begin{align*}
\mathbb{P}(S_{\Lambda_L(x)}(x) \geq (b\log N)^{\frac{1}{\alpha}}) 
\geq 
N^{
-(1+3\vare) \frac{c_\alpha}{(G^*)^\alpha}b}, 
\end{align*}
for every $N$ large enough. 


\end{proof}

\smallskip

\begin{proof}[Proof of Lemma \ref{a1} for the case 
$(A_\alpha)$\ $\alpha =1$]
By the assumption on $\eta$, if $|t|< c_\alpha$ then 
$\mathbb{E}[|\eta(0)|^k e^{t \eta(0)}]<\infty$ for 
every $k\in \mathbb{Z}_+$. 
Let $A\subset \Lambda_N$, $x\in \Lambda_{N}$ and $b>0$. 
Assume that $\lambda>0$ satisfies 
$\lambda \{\max\limits_{y\in A}G_N(x, y)\} < c_\alpha$. 
Then, we have 
\begin{align*}
\mathbb{P}(S_A(x) \geq b_N) 
& \leq 
e^{-\lambda b_N}
\mathbb{E}[e^{\lambda S_A(x)}] 
= 
e^{-\lambda b \log N}\prod_{y\in A} 
\mathbb{E}[e^{\lambda G_N(x, y)\eta(y)}]. 
\end{align*}
By Taylor's theorem, 
for every $u\in \mathbb{R}$ there exists 
$\rho(u) \in (0, 1)$ such that 
$e^u=1+u+ \frac{1}{2}u^2e^{\rho u}$. 
Therefore, 
\begin{align*}
\mathbb{E}[e^{\lambda G_N(x, y)\eta(y)}] 
& = 
\mathbb{E}\bigl[1+ 
\lambda G_N(x, y)\eta(y)+ 
\frac{1}{2}\lambda^2 G_N(x, y)^2 \eta(y)^2 
e^{\lambda G_N(x, y) \rho \eta(y)}\bigr] \\
& \leq 
1+ \frac{1}{2}\lambda^2 G_N(x, y)^2 
\mathbb{E}[\eta(y)^2 
e^{\lambda G_N(x, y) \rho \eta(y)}] \\
& \leq 
e^{C\lambda^2 G_N(x, y)^2}, 
\end{align*}
for some $C>0$. By these estimates, we obtain 
\begin{align*}
\mathbb{P}(S_A(x) \geq b_N) 
& \leq 
\exp \bigl\{ -\lambda b \log N + 
C\lambda^2 \sum_{y\in A} G_N(x, y)^2 \bigr\}. 
\end{align*}

Now we consider the case 
$A=\Lambda_N\cap (\Lambda_L(x))^c$. 
In this case, we have 
\begin{align*}
\sum\limits_{y\in \Lambda_N\cap (\Lambda_L(x))^c} G_N(x, y)^2 
\leq C \sum\limits_{r\geq L}r^{d-1}
(\frac{1}{r^{d-2}})^2 
\leq {C}{L^{-d+4}}. 
\end{align*}
By taking $\lambda$ as 
$\lambda=\frac{c_\alpha}{G_+}L^{\frac{d}{2}-2}$, we have 
$\lambda \{\max\limits_{y\in \Lambda_N \cap (\Lambda_L(x))^c}
G_N(x, y)\} < c_\alpha$ 
and 
$\lambda^2 
\sum\limits_{y\in \Lambda_N\cap (\Lambda_L(x))^c} G_N(x, y)^2 
\leq C<\infty$. 
Therefore, 
\begin{align*}
\mathbb{P}(S_{\Lambda_N\cap (\Lambda_L(x))^c(x)} \geq b \log N) 
& \leq 
C'' \exp \bigl\{ -C' L^{\frac{d}{2}-2} b \log N \bigr\}. 
\end{align*}
By combining this estimate with the union bound and Borel-Cantelli's 
lemma as before, we obtain (i). 

For the upper bound of 
$\mathbb{P}(S_{\Lambda_L(x)}(x) \geq b \log N)$, 
we take 
$\lambda$ as $\lambda=\frac{c_\alpha}{G^*}-\vare$. 
Then, 
$\max\limits_{y\in \Lambda_L(x)} \bigl\{
\lambda G_N(x, y) \bigr\} <c_{\alpha}$ and 
$
\lambda^2 \sum\limits_{y\in \Lambda_L(x)} 
G_N(x, y)^2 \leq C<\infty$. 
As a result, we obtain 
\begin{align*}
\mathbb{P}( 
S_{\Lambda_L(x)}(x) 
\geq b \log N ) 
& \leq 
C\exp\{- \lambda b \log N\} \leq 
N^{-(\frac{c_\alpha}{G^*} - 2\vare)b}. 
\end{align*}
The lower bound of the tail estimate follows from the same argument 
to the case $0<\alpha <1$. 
\end{proof}

\smallskip

\begin{proof}[Proof of Lemma \ref{a1} for the case 
$(A_\alpha)$\ $1<\alpha \leq 2$]
First of all, by the tail assumption on $\eta$ and 
Kasahara's Tauberian theorem (cf. \cite{K78}, see also 
\cite{bib5} p.253), 
$\log \mathbb{E}[e^{\lambda \eta(0)}] \sim {C}_{0} 
\lambda^{\overline{\alpha}}$ as $\lambda \to\infty$ 
where 
$\overline{\alpha}=\frac{\alpha}{\alpha-1}$ and 
$C_0= \bigl(\overline{\alpha}
(\alpha c_\alpha)^{\overline{\alpha}-1}\bigr)^{-1}$. 
In particular, for every $\vare >0$ 
there exists $T=T(\vare)>0$ such 
that if $\lambda \geq T$ then 
$\mathbb{E}[e^{\lambda \eta(0)}] \leq \exp\bigl\{ C_0(1+\vare) 
\lambda^{\overline{\alpha}}\bigr\}$. 
Without loss of generality we may assume that $T \geq 1$. 
Next, let $|\lambda| \leq T$. By Taylor's theorem 
and symmetry of $\eta$, 
\begin{align*}
\mathbb{E}[e^{\lambda \eta(0)}] 
& = 
1+ \lambda^2 \mathbb{E}\Bigl[
\sum\limits_{n\geq 2} \frac{\lambda^{n-2}}{n!}\eta(0)^n\Bigr] 
\leq 
1+ \lambda^2 \mathbb{E}\Bigl[
\sum\limits_{n\geq 0} \frac{T^{n}}{n!}\eta(0)^n\Bigr] 
= 1+\lambda^2 \mathbb{E}[e^{T\eta(0)}]. 
\end{align*}
Therefore, if 
$0\leq \lambda \leq T$ then 
$\mathbb{E}[e^{\lambda \eta(0)}] \leq 
e^{ {C_1} \lambda^{2}}$ 
where $C_1= \mathbb{E}[e^{T \eta(0)}]<\infty$. 
For $A\subset \Lambda_N$, $x\in \Lambda_{N}$ and $b>0$, 
we have 
\begin{align*}
\mathbb{P} ( S_{A}(x) \geq b_N ) 
& \leq e^{-\lambda_N b_N}\mathbb{E}
\bigl[e^{\lambda_N S_{A}(x)} \bigr] 
= e^{-\lambda_N b_N}
\prod_{y\in A}
\mathbb{E}\bigl[ e^{\lambda_N G_N(x, y) \eta(y)} \bigr], 
\end{align*}
where $\lambda_N>0$ is to be specified later on. 
We decompose the product into the following two parts: 
\noindent
\begin{align*}
A_1 & = \{y \in A ; \lambda_N G_N(x, y)\geq T \}, \quad 
A_2 = \{y \in A ; \lambda_N G_N(x, y)< T \}. 
\end{align*}
\\
For $A_1$, 
\begin{align*}
\prod\limits_{y \in A_1}
\Eb\bigl[ e^{\lambda_N G_N(x, y) \eta(y)} \bigr] 
& \leq 
\prod\limits_{y \in A_1}
\exp\Bigl\{ (1+\vare) C_0 \bigl(
\lambda_N G_N(x, y) \bigr)^{\overline{\alpha}} \Bigr\} \\
& = 
\exp\Bigl\{ (1+\vare) C_0 \lambda_N^{\overline{\alpha}}
\sum\limits_{y\in A_1} 
G_N(x, y)^{\overline{\alpha}} \Bigr\}. 
\end{align*}
For $A_2$, 
\begin{align*}
\prod\limits_{y \in A_2}
\Eb\bigl[ e^{\lambda_N G_N(x, y) \eta(y)} \bigr] 
& \leq 
\prod\limits_{y \in A_2}
\exp\Bigl\{  C_1 \bigl(
\lambda_N G_N(x, y) \bigr)^{2} \Bigr\} \\
& = 
\exp\Bigl\{ C_1 \lambda_N^{2}
\sum\limits_{y\in A_2} 
G_N(x, y)^{2} \Bigr\}. 
\end{align*}

In the case 
$A= \Lambda_{N} \cap (\Lambda_L(x))^c $, we have 
\begin{align}\label{720}
\sum\limits_{y\in A_1} G_N(x, y)^{\overline{\alpha}} 
& \leq 
\sum\limits_{y\in \Lambda_L(x)^c} 
\bigl(\frac{G_+}{|x-y|^{d-2}}\bigr)^{\overline{\alpha}} 
\leq 
C \sum\limits_{r\geq L}r^{d-1} 
\bigl(\frac{1}{r^{d-2}}\bigr)^{\overline{\alpha}} 
\leq 
CL^{-\frac{d-2\alpha}{\alpha-1}}, 
\end{align}
and
\begin{align*}
\sum\limits_{y\in A_2} G_N(x, y)^{2} 
& \leq 
\sum\limits_{y\in \Lambda_L(x)^c} 
\bigl(\frac{G_+}{|x-y|^{d-2}}\bigr)^{2} 
\leq 
C \sum\limits_{r\geq L}r^{d-1} 
\bigl(\frac{1}{r^{d-2}}\bigr)^{2} 
\leq 
CL^{-{d+4}}, 
\end{align*}
for some constant $C>0$. 
We remark that the last inequality in (\ref{720}) 
holds under the condition $d>2\alpha$. 
Therefore, 
\begin{align*}
\mathbb{P} \bigl( S_{\Lambda_N\cap (\Lambda_{L}(x))^c}(x)
\geq b_N \bigr) 
& \leq \exp\bigl\{-\lambda_N b_N 
+ C \lambda_N^{\overline{\alpha}} L^{-\frac{d-2\alpha}{\alpha-1}}+
C \lambda_N^{2} L^{-d+4}\bigr\}. 
\end{align*}
Since $\overline{\alpha}\geq 2$ in this case, 
by taking $\lambda_N$ so that $\lambda_N \to \infty$ as $N\to \infty$ 
we have 
\begin{align*}
\mathbb{P} \bigl( S_{\Lambda_N\cap (\Lambda_{L}(x))^c}(x) 
\geq b_N \bigr) 
\leq 
\exp\Bigl\{-\lambda_N b_N
+ C' \lambda_N^{\overline{\alpha}} L^{-\frac{d-2\alpha}{\alpha-1}}\Bigr\}, 
\end{align*}
for every $N$ large enough. 
Then, 
$\lambda_N = \bigl\{\frac{b_N}{C'\overline{\alpha}}
L^{\frac{d-2\alpha}{\alpha-1}}\bigr\}^{\alpha-1}$ 
optimizes the right hand side and we obtain 
\begin{align*}\label{ub}
\mathbb{P} \bigl( S_{\Lambda_N\cap (\Lambda_{L}(x))^c}(x) 
\geq b_N \bigr) 
\leq 
\exp\bigl\{-C''L^{d-2\alpha} (b_N)^{\alpha}\bigr\} 
= N^{- C''  L^{d-2\alpha} b}, 
\end{align*}
for every $x\in \Lambda_{N}$ and $N$ large enough. 
(i) follows from this estimate as before. 

Next, let $L\in \mathbb{N}$ be fixed and 
and consider the tail estimate 
$\mathbb{P} ( S_{\Lambda_{L}(x)}(x) \geq  b_N )$. 
By taking $\lambda_N$ so that $\lambda_N\to \infty$ as 
$N\to \infty$, 
we may assume that $\lambda_N G_N(x, y) \geq T$ 
for every $y\in \Lambda_L(x)$. 
Therefore,  
\begin{align*}
\mathbb{P} \bigl( S_{\Lambda_{L}(x)}(x) 
\geq b_N \bigr) 
\leq 
\exp\Bigl\{-\lambda_N b_N+ (1+\vare) C_0 
\lambda_N^{\overline{\alpha}} 
\sum\limits_{y\in \Lambda_L(x)} 
G_N(x, y)^{\overline{\alpha}}\Bigr\}. 
\end{align*}
Since $G_N(x, y) \to G(x, y)$ 
as $N \to \infty$ for every $x, y\in \mathbb{Z}^d$, 
by optimizing the right hand side in $\lambda_N$ 
we obtain the following for every $N$ large enough: 
\begin{align*}
\mathbb{P} \bigl( S_{\Lambda_{L}(x)}(x) 
\geq b_N \bigr) 
\leq 
\exp\Bigl\{- (1-\vare) \frac{c_\alpha}
{(G_{L, (\alpha)}^*)^{\alpha-1}} b_N^\alpha \Bigr\}
= 
N^{-(1-\vare)\frac{c_\alpha}{(G_{L, (\alpha)}^*)^{\alpha-1}}b}, 
\end{align*}
where 
$G_{L, (\alpha)}^*:=
\sum\limits_{y \in \Lambda_L}G(0, y)^{\overline{\alpha}}$. 

Finally, for the lower bound of the tail estimate we have 
\begin{align*}
\mathbb{P} \bigl( S_{\Lambda_{L}(x)}(x) 
\geq b_N \bigr) 
& = 
\mathbb{P} \Bigl( \sum\limits_{y\in \Lambda_{L}(x)}G_N(x, y) \eta(y) 
\geq \sum\limits_{y\in \Lambda_{L}(x)}
\bigl\{
\sum\limits_{y\in \Lambda_{L}(x)}G_N(x, y)^{\overline{\alpha}}
\bigr\}^{-1}
G_N(x, y)^{\overline{\alpha}}b_N \Bigr) \\
& \geq 
\mathbb{P} \Bigl( \eta(y) \geq 
\bigl\{
\sum\limits_{y\in \Lambda_{L}(x)}G_N(x, y)^{\overline{\alpha}}
\bigr\}^{-1}
G_N(x, y)^{\overline{\alpha}-1}b_N 
\text{ for every } y\in \Lambda_L(x)\Bigr) \\
& = 
\prod_{y\in \Lambda_L(x)} 
\mathbb{P} \Bigl( \eta(y) \geq 
\bigl\{
\sum\limits_{y\in \Lambda_{L}(x)}G_N(x, y)^{\overline{\alpha}}
\bigr\}^{-1}
G_N(x, y)^{\overline{\alpha}-1}b_N 
\Bigr) \\
& \geq 
\prod_{y\in \Lambda_L(x)} 
\exp\Bigl\{ -(1+\vare)c_\alpha 
\bigl\{\frac{G_N(x, y)^{\overline{\alpha}-1}}
{\sum\limits_{y\in \Lambda_{L}(x)}G_N(x, y)^{\overline{\alpha}}}
\bigr\}^{\alpha} (b_N)^\alpha\Bigr\}\\
& \geq 
\exp\Bigl\{-(1+2\vare)c_\alpha \frac{1}
{(G_{L, (\alpha)}^*)^{{\alpha}-1}}b \log N
\Bigr\}, 
\end{align*}
for every $N$ large enough 
where the second last inequality follows from the tail assumption on $\eta$ 
and we used the fact that 
$G_N(x, y) \to G(x, y)$ 
as $N \to \infty$ for every $x, y\in \mathbb{Z}^d$
for the last inequality. 
\end{proof}

\smallskip

\begin{proof}[Proof of Lemma \ref{a1} for the case 
$(\widetilde{A})$]
By the similar argument to the proof for the case 
$(A_\alpha)$\ $1 < \alpha \leq 2$, 
we can prove that there exists 
$L_0=L(\delta)\in \mathbb{N}$ such that 
\begin{equation*}\label{426}
\mathbb{P}\bigl( 
|S_{\Lambda_N\cap (\Lambda_{L_0}(x))^c}(x)| \leq 
\frac{1}{2}(\delta \log N)^{\frac{1}{2}} \text{ for every } 
x \in \Lambda_{N}, \text{ for all large } N 
\bigr)=1.
\end{equation*}
Hence, it is sufficient to prove 
\begin{equation}\label{503}
\mathbb{P}\bigl( 
|S_{\Lambda_N \cap \Lambda_{L_0}(x)}(x)| \leq 
\frac{1}{2}(\delta \log N)^{\frac{1}{2}} \text{ for every } 
x \in \Lambda_{N}, \text{ for all large } N 
\bigr)=1.
\end{equation}
If $|\eta(y)| <\frac{(\delta \log N)^{\frac{1}{2}}}
{2|\Lambda_{L_0}|G^*}$ for every $y\in \Lambda_N$, then 
$|\sum\limits_{y\in \Lambda_N \cap \Lambda_{L_0}(x)} 
G_N(x, y) \eta(y)| <
\frac{1}{2}{(\delta \log N)^{\frac{1}{2}}}$ 
for every $x\in \Lambda_{N}$. Therefore, by the tail assumption on 
$\eta$ we have 
\begin{align*}
\mathbb{P}\bigl( 
|S_{\Lambda_N \cap \Lambda_{L_0}}(x)| \geq \frac{1}{2} 
(\delta {\log N})^{\frac{1}{2}} \text{ for some } 
x \in \Lambda_{N}\bigr) 
& \leq 
\mathbb{P}\bigl( 
|\eta(y)| \geq 
\frac{(\delta \log N)^{\frac{1}{2}}}{2|\Lambda_{L_0}|G^*} 
\text{ for some } y\in \Lambda_N \bigr) \\
& \leq 
CN^d \exp\Bigl\{ -C f(N)(\log N)\Bigr\}, 
\end{align*}
for every $N$ large enough where $f(N)$ satisfies 
$f(N)\to \infty$ as $N\to \infty$. 
The right hand side is summable in $N\geq 1$ 
and (\ref{503}) follows from Borel-Cantelli's lemma. 
\end{proof}

\section{Proof of Theorem \ref{thm1}}
In this section we give the proof of Theorem \ref{thm1}. 
By the observation given in subsection 1.2, the law of 
$\{\phi(x)\}_{x\in \Lambda_N}$ under $\mu_N^{\eta}$ 
for fixed $\eta=\{\eta(x)\}_{x\in \mathbb{Z}^d}$ 
is the same as the law of 
$\{{\phi}(x)+ m_N^{\eta}(x) \}_{x\in \Lambda_N}$ where 
$\{{\phi(x)}\}_{x\in \Lambda_N}$ is distributed by $\mu_N$ and 
is independent of $\eta$. 
$m_N^\eta$ is given by 
$m_N^\eta(x)= E^{\mu_N^{\eta}}[\phi(x)] 
= \sum\limits_{y\in \Lambda_N}G_N(x, y)\eta(y)$, $x\in \Lambda_N$.


\subsection{Lower bound of the maximum }
For the proof of (\ref{maxglb}), we have 
\begin{equation*}
\mu_N^{\eta}\bigl( 
\max\limits_{x\in \Lambda_N} \phi(x) \leq 
b(\log N)^{\frac{1}{\alpha \wedge 2}} \bigr) 
=
\mu_N\bigl( \max\limits_{x\in \Lambda_N} 
\{\phi(x) + m_N^\eta(x)\} \leq 
b(\log N)^{\frac{1}{\alpha \wedge 2}}\bigr), 
\end{equation*}
for every $b\in \mathbb{R}$. 
By Lemma \ref{a1} we may replace 
$m_N^\eta(x)= \sum\limits_{y\in \Lambda_N}G_N(x, y)\eta(y)$ 
by $\widetilde{m}_{N, L}^\eta(x) := 
\sum\limits_{y\in \Lambda_L(x)}G_N(x, y)\eta(y)$ 
for $L$ large enough. 
We apply the quenched-annealed comparison argument used in 
\cite{bib3} and \cite{bib4}. 
For this purpose we prepare several notations. 
$\mathbb{P} \otimes \mu_N$ denotes the product measure 
on the configuration space 
$\mathbb{R}^{\mathbb{Z}^d}\times \mathbb{R}^{\Lambda_N}$ 
equipped with product topology and Borel $\sigma$-algebra. 
We may identify the configuration 
$\eta=\{\eta(x)\}_{x\in \mathbb{Z}^d}$ as a family of 
i.i.d. random variables under $\mathbb{P}$ 
by considering the coordinate map.  
$\mathcal{T}^{\eta}$ denotes 
trivial $\sigma$-algebra for $\eta$-field. 
We also define $\sigma$-algebras for $\eta$-field 
and $\phi$-field 
on $A\subset \mathbb{Z}^d$ 
by 
$\mathcal{F}^\eta_A = \sigma( \eta(x); x \in A)$ 
and 
$\mathcal{F}^\phi_A = \sigma( \phi(x); x \in A)$, 
respectively. 
For an $\mathcal{F}^\phi_{\Lambda_N}$-measurable event $\mathcal{A}$, 
we sometimes interpret it as 
$\mathbb{R}^{\mathbb{Z}^d} \times \mathcal{A}$ 
for notational simplicity. 


Now, we consider the $\mathcal{F}_{\mathbb{Z}^d}^\eta$-measurable 
random variable: 
\begin{equation*}
X_N:= \mu_N\bigl( \max\limits_{x\in \Lambda_N} 
\{\phi(x) + \widetilde{m}_{N, L}^\eta(x)\} \leq 
b(\log N)^{\frac{1}{\alpha \wedge 2}}\bigr),
\end{equation*}
and assume that we have the estimate 
${\mathbb{E}}[X_N] \leq e^{-N^\lambda}$ 
for some $\lambda >0$ and every $N$ large enough. 
Then, by Markov's inequality we have 
$
{\mathbb{P}} ( X_N \geq e^{-(1-\vare)N^\lambda} ) 
\leq e^{-\vare N^\lambda}$ 
for $\vare >0 $. 
The right hand side is summable in $N \geq 1$ and 
Borel Cantelli's lemma yields 
${\mathbb{P}} (X_N \leq e^{-(1-\vare)N^\lambda}, 
\text{ for all large } N )=1$. 
Hence we obtain 
$\lim\limits_{N\to \infty}
\mu_N^{\eta}\bigl( \max\limits_{x\in \Lambda_N} \phi(x) 
\geq b(\log N)^{\frac{1}{\alpha \wedge 2}} \bigr)=1$, 
${\mathbb{P}}$-a.s..
Since we assume that the law of $\eta$ is symmetric, 
${\mathbb{E}} [X_N]= {\mathbb{P}}\otimes \mu_N
(\widetilde{\Omega}_{N}^+(-b))$ 
by Fubini's theorem where 
\begin{equation*}
\widetilde{\Omega}_{N}^+(a) := \{ (\eta, \phi); 
\phi(x) + \widetilde{m}_{N, L}^\eta(x) 
\geq a (\log N)^{\frac{1}{\alpha \wedge 2}}, 
\text{ for every }x \in \Lambda_{N} \},\ a\in \mathbb{R}. 
\end{equation*}
Therefore, for the proof of (\ref{maxglb}) 
it is sufficient to prove that 
if $b< M^*$, then 
there exists some constant $\lambda >0$ 
such that 
$
{\mathbb{P}}\otimes \mu_N
(\widetilde{\Omega}_{N}^+(-b))
\leq e^{-N^\lambda}$ 
for every $N$ large enough 
where $M^*$ is given by (\ref{MM}). 

For $1 \ll L \ll {N}$, set
$\Gamma_N =\{ x \in 4L\mathbb{Z}^d; \Lambda_L(x) \subset \Lambda_{N}\}$ 
and $D_L = \bigcup\limits_{x \in \Gamma_N} \partial^+ \Lambda_L(x)$ 
where $\Lambda_L(x) = x + \Lambda_L$ for $x \in \mathbb{Z}^d$. 
We also define 
$M_L^\phi(x) = E^{\mu_N}\bigl[\phi(x) | \mathcal{F}_{D_L}^\phi
\bigr]$ for $x \in \Gamma_N$. 
Note that $M_L^\phi(x)$ is 
$\sigma(\phi(z); z\in \partial^+ \Lambda_L(x))$-measurable for 
each $x \in \Gamma_N$ by spatial Markov property of the field. 
For $\delta, \delta' \in (0, 1)$ and $\beta, \beta' \in \mathbb{R}$, 
consider the events: 
\begin{align*}
\mathcal{A}_{\delta, \beta} & = 
\bigl\{ \phi ; | \{x \in {\Gamma}_N ; 
M_L^\phi(x) \leq \beta(\log N)^{\frac{1}{\alpha \wedge 2}}\}| 
\leq \delta |{\Gamma}_N| \bigr\},
\\
\mathcal{A}'_{\delta', \beta'} & = 
\bigl\{ \phi ; | \{x \in {\Gamma}_N ; 
\phi(x) \leq \beta' (\log N)^{\frac{1}{\alpha \wedge 2}} 
\}| \leq \delta' |{\Gamma}_N| \bigr\}, 
\end{align*}
where for every $A\subset\mathbb{Z}^d$, 
$|A|$ denotes its cardinality. 
Then, we have 
\begin{equation*}
\widetilde{\Omega}_{N}^+(-b) \subset 
\bigl(\widetilde{\Omega}_{N}^+(-b) \cap (\mathcal{A}_{\delta, \beta})^c \bigr)
\cup 
\bigl(\mathcal{A}_{\delta, \beta} \cap (\mathcal{A}'_{2\delta, \beta'})^c 
\bigr) \cup 
\bigl(\mathcal{A}'_{2\delta, \beta'} \cap 
\widetilde{\Omega}_{N}^+(-b) \bigr).
\end{equation*}
We estimate the probability of each term in the right hand side. 

For the first term, 
\begin{align*}
& \mathbb{P} \otimes \mu_N\bigl( 
\widetilde{\Omega}_{N}^+(-b) \cap (\mathcal{A}_{\delta, \beta})^c \bigr) \\
& \quad \leq 
E^{\mathbb{P}\otimes \mu_N}\Bigl[
\prod_{x \in \Gamma_N} \mathbb{P}\otimes \mu_N\bigl(\phi(x)
+ \widetilde{m}_{N, L}^\eta(x) \geq 
-b (\log N)^{\frac{1}{\alpha \wedge 2}} \bigm| 
\mathcal{T}^\eta \otimes \mathcal{F}_{D_L}^\phi \bigr); 
(\mathcal{A}_{\delta, \beta})^c \Bigr]\\
& \quad = 
E^{\mathbb{P}\otimes \mu_N}\Bigl[
\prod_{x \in \Gamma_N} \Bigl\{1- 
\mathbb{P}\otimes \mu_N\bigl(\phi(x)+\widetilde{m}_{N, L}^\eta(x)
 < -b (\log N)^{\frac{1}{\alpha \wedge 2}} \bigm| 
\mathcal{T}^\eta \otimes \mathcal{F}_{D_L}^\phi 
\bigr) \Bigr\}; 
(\mathcal{A}_{\delta, \beta})^c \Bigr]\\
& \quad \leq 
E^{\mathbb{P}\otimes \mu_N}\Bigl[
\prod_{x \in \Gamma_N} \Bigl\{1- 
\mu_N\bigl(\phi(x)  \leq 
-\gamma (\log N)^{\frac{1}{\alpha \wedge 2}} 
\bigm| \mathcal{F}^\phi_{D_L} \bigr) \\
& \qquad \qquad \qquad \qquad \qquad \qquad \times 
\mathbb{P}\bigl(-\widetilde{m}_{N, L}^\eta(x) \geq 
(b-\gamma) (\log N)^{\frac{1}{\alpha \wedge 2}}\bigr) \Bigr\} ; 
(\mathcal{A}_{\delta, \beta})^c \Bigr], 
\end{align*}
for every $\gamma \in \mathbb{R}$. 
Note that $\{\widetilde{m}_{N, L}^\eta(x)\}_{x\in \Gamma_N}$ 
is a family of independent random variables. 
The estimate on the right hand side changes depending on the 
tail assumption on $\eta$. 
We remark that by Proposition \ref{gfunc} and the 
similar computation as (\ref{720}),  
$G^*_{(\alpha)}=\sum\limits_{x\in \mathbb{Z}^d} 
G(0, x)^{\overline{\alpha}}<\infty$ if $d>2\alpha$ and 
$d\geq 3$. In particular 
$G^{*}_{L, (\alpha)} \to G^*_{(\alpha)}$ as $L\to \infty$ 
in this case. 

\smallskip

\noindent
{\underline{{\em The case $(A_\alpha)$\ $\alpha =2$}}}

\noindent
Let $\vare>0$, $\beta >0$ and $0<\gamma<b$. 
For each $x \in \Gamma_N$, if $M_L^\phi(x) \leq \beta \sqrt{\log N}$ 
then we have 
\begin{align*}
\mu_N\bigl(\phi(x) \leq - \gamma \sqrt{\log N} \bigm| 
\mathcal{F}^\phi_{D_L} \bigr) 
& \geq 
\mu_N\bigl(\phi(x) -M_L^\phi(x) \leq 
-(\beta+\gamma) \sqrt{\log N}\, \bigm| 
\mathcal{F}^\phi_{D_L} \bigr) \\
& \geq 
\exp\Bigl\{-\frac{1+\vare}{2G_L^*}(\beta+\gamma)^2\log N\Bigr\} \\
& = 
N^{-\frac{1+\vare}{2G_L^*}(\beta+\gamma)^2}, 
\end{align*}
for every $N$ large enough. 
Note that the law of $\phi(x)-M_L^\phi(x)$ under 
$\mu_N(\, \cdot \, | \mathcal{F}_{D_L}^\phi)$ is centered 
Gaussian with the variance $G_L^*:=G_L(0, 0)$.  
Also, by Lemma \ref{a1}\, (ii), 
\begin{equation*}
\mathbb{P}\bigl(-\widetilde{m}_{N, L}^\eta(x) \geq 
(b-\gamma) \sqrt{\log N}\bigr) \geq N^{-(1+\vare )K_L^*(b-\gamma)^2},  
\end{equation*}
for every $x \in \Gamma_N$ 
where 
$K_L^*= \frac{c_2}{G_{L, (2)}^*}$, 
${G_{L, (2)}^*}=\sum\limits_{x\in \Lambda_L} 
G(0, x)^{2}$. 
By collecting all the estimates, we have 
\begin{align}\label{a73}
\begin{split}
\mathbb{P} \otimes \mu_N\bigl( 
\widetilde{\Omega}_{N}^+(-b) \cap (\mathcal{A}_{\delta, \beta})^c \bigr) 
& \leq 
\Bigl\{1-N^{-(1+\vare)\bigl( \frac{1}{2G_L^*}(\beta+\gamma)^2
+K_L^*(b-\gamma)^2\bigr) }\Bigr\}^{\delta |\Gamma_N|} \\
& \leq 
\exp\Bigl\{
-CN^{d-(1+\vare)\bigl(\frac{1}{2G_L^*}(\beta+\gamma)^2
+K_L^*(b-\gamma)^2\bigr)}\Bigr\}, 
\end{split}
\end{align}
for every $N$ large enough. 
Now, we take $\gamma$ as 
$\gamma= \frac{2G_L^*K_L^* b}{2G^*_L K_L^*+1}$. 
Then, 
$\frac{1}{2G_L^*}\gamma^2
+K_L^*(b-\gamma)^2 
= \frac{1}{2G_L^*+ \frac{1}{K_L^*}}b^2$. 
Since 
$G_L^* \to G^*$, 
$G_{L, (2)}^* \to G_{(2)}^*$ as $L\to \infty$ 
and we assume that 
$b<\sqrt{(2G^*+\frac{G_{(2)}^*}{c_2})d}$, 
if we take $\vare>0$ small enough 
and $L$ large enough then there exists 
$\beta=\beta_0>0$ small enough such that 
the right hand side of (\ref{a73}) 
is less than $e^{-N^\lambda}$ for 
some $\lambda>0$. 

\smallskip

\noindent
{\underline{{\em The case $(A_\alpha)$ $0< \alpha <2$}}}

\noindent
In this case we take $\gamma$ as $\gamma = -\beta$ 
in the above argument. Then, 
\begin{align*}
\mu_N\bigl(\phi(x) \leq -\gamma ({\log N})^{\frac{1}{\alpha}} 
\bigm| \mathcal{F}^\phi_{D_L}\bigr) 
\geq 
\mu_N\bigl(\phi(x) - M_L^\phi(x)\leq 0 
\bigm| \mathcal{F}^\phi_{D_L} \bigr) = \frac{1}{2}, 
\end{align*}
under the condition $M_L^\phi(x) \leq \beta 
(\log N)^{\frac{1}{\alpha}}$ and 
\begin{align*}
\mathbb{P} \bigl(-\widetilde{m}_{N, L}^\eta(x) \geq 
(b-\gamma) ({\log N})^{\frac{1}{\alpha}}\bigr) 
\geq N^{-(1+\vare)K_L^*(b+\beta)^\alpha}, 
\end{align*}
where $K_L^*$ is given by (\ref{constK}). 
Therefore, 
\begin{align*}
\mathbb{P} \otimes \mu_N\bigl( 
\widetilde{\Omega}_{N}^+(-b) \cap (\mathcal{A}_{\delta, \beta})^c \bigr) 
& \leq 
\Bigl\{1-\frac{1}{2}N^{-(1+\vare)K_L^*(b+\beta)^\alpha}
\Bigr\}^{\delta |\Gamma_N|} \\
& \leq 
\exp\Bigl\{
-CN^{d-{(1+\vare)}K_L^*(b+\beta )^\alpha}
\Bigr\}. 
\end{align*}
If $b< M^*$ 
then by taking $\vare>0$ small enough and $L$ 
large enough, there exists $\beta=\beta_0>0$ small enough such that 
the right hand side is less than $e^{-N^\lambda}$ for 
some $\lambda>0$. 

\smallskip

\noindent
{\underline{{\em The case $(\widetilde{A})$}}}

\noindent
In this case we take $\gamma$ as $\gamma =b$. 
Then, 
$\mathbb{P} \bigl(-\widetilde{m}_{N, L}^\eta(x) \geq 
(b-\gamma) \sqrt{\log N}\bigr) \geq \frac{1}{2}$ and 
we have the estimate 
\begin{align*}
\mathbb{P} \otimes \mu_N\bigl( 
\widetilde{\Omega}_{N}^+(-b) \cap (\mathcal{A}_{\delta, \beta})^c \bigr) 
& \leq 
\Bigl\{1-\frac{1}{2}N^{-\frac{(1+\vare)}{2G_L^*}(\beta + b)^2}
\Bigr\}^{\delta |\Gamma_N|} \\
& \leq 
\exp\Bigl\{
-CN^{d-\frac{(1+\vare)}{2G_L^*}(\beta + b)^2}
\Bigr\}. 
\end{align*}
If $b<\sqrt{2dG^*}$, then by taking $\vare>0$ small enough 
and $L$ large enough, 
there exists $\beta=\beta_0>0$ small enough such that 
the right hand side is less than $e^{-N^\lambda}$ for 
some $\lambda>0$. 

\smallskip

The rest of the proof works well in all cases. 
Let $\beta' \in (0, \beta_0)$ 
where $\beta_0>0$ is given by the above argument. 
On $\mathcal{A}_{\delta, \beta_0} \cap 
(\mathcal{A}'_{2\delta, \beta'})^c$ we have 
$ | \{x \in {\Gamma}_N; 
\phi(x)-M_L^\phi(x)< -({\beta_0}-{\beta'})
(\log N)^{\frac{1}{\alpha \wedge 2}} \}| 
\geq \delta |{\Gamma}_N|$. 
Since $\{\phi(x)- M_L^\phi(x)\}_{x\in \Gamma_N}$ under 
$\mu_N(\ \cdot\ | \mathcal{F}_{D_L}^\phi )$ 
is a family of i.i.d. random variables and 
$\mu_N( \phi(x)-M_L^\phi(x)< -({\beta_0}-{\beta'})
(\log N)^{\frac{1}{\alpha \wedge 2}} |\, 
\mathcal{F}_{D_L}^\phi ) \to 0$ 
as $N\to \infty$, 
by standard LDP argument we have 
$\mathbb{P}\otimes \mu_N\bigl(
\mathcal{A}_{\delta, \beta_0} \cap (\mathcal{A}'_{2\delta, \beta'})^c 
\bigr)\leq e^{-CN^d}$ 
for every $N$ large enough. 

For the estimate of 
$
\mathbb{P}\otimes \mu_N\bigl(
\mathcal{A}'_{2\delta, \beta'} \cap 
\widetilde{\Omega}_{N}^+(-b)\bigr)
= 
{\mathbb{E}}\bigl[
\mu_N\bigl( \mathcal{A}'_{2\delta, \beta'} \cap 
\widetilde{\Omega}_{N}^+(-b) \bigr) \bigr]$, 
we consider the $\mathcal{F}^\eta_{\mathbb{Z}^d}$-measurable event: 
\begin{equation*}
\mathcal{B}_{\beta'', \delta}= 
\bigl\{\eta ; | \{x \in {\Gamma}_N ; 
\widetilde{m}_{N, L}^\eta(x) \geq \beta'' 
(\log N)^{\frac{1}{\alpha \wedge 2}}\} | \geq \delta 
|{\Gamma}_N| \bigr\}, 
\end{equation*}
for $\beta'', \delta>0$. 
Since $\{\widetilde{m}_{N, L}^\eta(x)\}_{x \in \Gamma_N}$ 
are independent and 
$\max\limits_{x\in \Gamma_N} 
\mathbb{P}( \widetilde{m}_{N, L}^\eta(x) \geq \beta'' 
(\log N)^{\frac{1}{\alpha \wedge 2}}
)\to 0$ as $N\to \infty$, 
by standard  LDP estimate again we have 
$\mathbb{P} (\mathcal{B}_{\beta'', \delta}) \leq 
e^{-CN^d}$. 
Therefore, we have only to consider the expectation on the event 
$(\mathcal{B}_{\beta'', \delta})^c$. 
We fix a realization 
$\eta \in (\mathcal{B}_{\beta'', \delta})^c$ 
and set 
$\widetilde{\Gamma}_N = \{x \in \Gamma_N; 
\widetilde{m}_{N, L}^\eta(x) \leq \beta'' 
(\log N)^{\frac{1}{\alpha \wedge 2}}\}$. 
For 
$\phi \in 
\mathcal{A}'_{2\delta, \beta'} \cap 
\widetilde{\Omega}_{N}^+(-b)$ 
there are at least $(1-3\delta)|\widetilde{\Gamma}_N|$ 
sites $x \in \widetilde{\Gamma}_N$ such that 
$\phi(x) \geq \beta'({\log N})^{\frac{1}{\alpha \wedge 2}}$ 
and in the remaining 
at most $3\delta|\widetilde{\Gamma}_N|$
sites of $\widetilde{\Gamma}_N$ we have 
$\phi(x) \geq -\widetilde{m}_{N, L}^\eta(x) 
-b(\log N)^{\frac{1}{\alpha \wedge 2}} 
\geq -(\beta''+b) (\log N)^{\frac{1}{\alpha \wedge 2}}$. 
Then, 
$
\sum\limits_{x \in \widetilde{\Gamma}_N} \phi(x) 
\geq \bigl\{(1-3\delta)\beta' - 3\delta
(\beta''+b)\bigr\}|\widetilde{\Gamma}_N|
(\log N)^{\frac{1}{\alpha \wedge 2}}$ 
and we can take $\delta, \beta''>0$ small enough so that 
$\lambda_1:= 
(1-3\delta)\beta' - 3\delta (\beta''+b)>0$. 
As a result, for given 
$\eta \in (\mathcal{B}_{\beta'', \delta})^c$ it holds that 
\begin{align*}
\mu_N\bigl(
\mathcal{A}'_{2\delta, \beta'} \cap 
\widetilde{\Omega}_{N}^+(-b) \bigr) 
& \leq 
\mu_N\bigl(
\sum\limits_{x \in \widetilde{\Gamma}_N} \phi(x) 
\geq \lambda_1 |\widetilde{\Gamma}_N| 
(\log N)^{\frac{1}{\alpha \wedge 2}} \bigr) \\
& \leq 
\exp\Bigl\{
-\frac{\lambda_1^2 |\widetilde{\Gamma}_N|^2 
(\log N)^{\frac{2}{\alpha \wedge 2}}}
{2\mathrm{Var}_{\mu_N}\bigl(
\sum\limits_{x \in \widetilde{\Gamma}_N} 
\phi(x)\bigr)}\Bigr\}. 
\end{align*}
Though $\widetilde{\Gamma}_N$ is an 
$\mathcal{F}^\eta_{\mathbb{Z}^d}$-measurable random set, 
we have uniform estimates: 
\begin{align*}
\mathrm{Var}_{\mu_N}\bigl(
\sum\limits_{x \in \widetilde{\Gamma}_N} 
\phi(x)\bigr) \leq 
\mathrm{Var}_{\mu_N}\bigl(
\sum\limits_{x \in {\Gamma}_N} 
\phi(x)\bigr) \leq CN^{d+2}, 
\end{align*}
and 
$|\widetilde{\Gamma}_N|\geq (1-\delta)|{\Gamma}_N| \geq CN^d$. 
Hence, we obtain 
\begin{align*}
\mathbb{P}\otimes \mu_N\bigl(
\mathcal{A}'_{2\delta, \beta'} \cap 
\widetilde{\Omega}_{N}^+(-b)\bigr)
& \leq 
{\mathbb{E}}\bigl[
\sup\limits_{\eta \in (\mathcal{B}_{\beta'', \delta})^c}
\mu_N\bigl(
\mathcal{A}'_{2\delta, \beta'} \cap 
\widetilde{\Omega}_{N}^+(-b) \bigr)\bigr] \\
& \leq e^{-CN^{d-2}(\log N)^{\frac{2}{\alpha \wedge 2}}}, 
\end{align*} 
and this completes the proof of (\ref{maxglb}). 

\qed

\subsection{Upper bound of the maximum}
First of all, for the proof of (\ref{maxgub}) and (\ref{maxgub2}) 
we give estimates on the number of high points 
of $m_N^\eta$-field. 

\smallskip

\begin{lemma}\label{mht}
Let $d\geq 5$ and 
set 
$m_N^\eta(x)= \sum\limits_{y\in \Lambda_N}G_N(x, y) \eta(y)$, 
$x\in \Lambda_N$. 
\begin{enumerate}[$(i)$]
\item
Assume that 
$\eta=\{\eta(x)\}_{x\in \mathbb{Z}^d}$ is a family of i.i.d. 
symmetric random variables which satisfies the 
condition $(A_\alpha)$. 
Then, for every $b\in (0, \frac{d}{K}]$ and $\vare>0$ 
it holds that 
\begin{equation}\label{151}
\mathbb{P}\bigl( 
|\{x\in \Lambda_{N}; m_N^\eta(x) \geq 
({b \log N})^{\frac{1}{\alpha}}\}| 
\leq N^{d-Kb+\vare}, 
\text{ for all large } N \bigr)=1, 
\end{equation}
where $K$ is a constant given by $(\ref{constK})$. 
Also, for every $\vare>0$ we have 
\begin{equation}\label{152}
\mathbb{P}\bigl( 
\max\limits_{x\in \Lambda_N} m_N^\eta(x) \leq 
(\frac{d+1+\vare}{K} \log N)^{\frac{1}{\alpha}}, 
\text{ for all large } N \bigr)=1. 
\end{equation}
\item
Assume that 
$\eta=\{\eta(x)\}_{x\in \mathbb{Z}^d}$ is a family of i.i.d. 
symmetric random variables which satisfies the 
condition $(\widetilde{A})$. 
Then, for every $\vare >0$ we have 
\begin{equation*}
\mathbb{P}\bigl( 
\max_{x\in \Lambda_N} m_N^\eta(x) \leq 
\vare \sqrt{ \log N}, \text{ for all large } N 
\bigr)=1.
\end{equation*}
\end{enumerate}
\end{lemma}

\begin{proof}
By Lemma \ref{a1} (i), we may replace $m_N^\eta(x)$ by 
$\widetilde{m}_{N, L}^\eta(x):= \sum\limits_{y\in \Lambda_L(x)} 
G_N(x, y) \eta(y)$ for fixed $L$ large enough. 
Set $\widetilde{\Lambda}_{L}:= [0, 4L-1]^d\cap \mathbb{Z}^d$ 
and $\Lambda_{N}^L(z):= \Lambda_{N}\cap 
(z+4L\mathbb{Z}^d)$. Then, $\Lambda_{N}$ can be 
decomposed as $\Lambda_{N}= 
\bigcup\limits_{z\in \widetilde{\Lambda}_{L}}
\Lambda_{N}^L(z)$ and we have 
\begin{align*}
& \mathbb{P}\bigl( 
|\{x\in \Lambda_{N}; \widetilde{m}_{N, L}^\eta(x) 
\geq (b \log N)^{\frac{1}{\alpha}}\}| 
\geq N^{d-Kb+\vare} \bigr) \\
& \quad \leq 
\sum\limits_{z\in \widetilde{\Lambda}_{L}}
\mathbb{P}\bigl( 
|\{x\in \Lambda_{N}^L(z) ; 
\widetilde{m}_{N, L}^\eta(x) \geq 
(b \log N)^{\frac{1}{\alpha}}\}| 
\geq \frac{1}{|\widetilde{\Lambda}_{L}|} N^{d-Kb+\vare} \bigr).
\end{align*}
For each $z \in \widetilde{\Lambda}_{L}$, 
$\{\widetilde{m}_{N, L}^\eta(x); x \in 
\Lambda_{N}^L(z)\}$ is a family of independent random variables. 
We recall Bernstein's concentration inequality (cf. \cite{bib2}). 
Let $\{X_i\}_{i\geq 1}$ be a family of independent random variables 
which satisfies $|X_i| \leq b$ a.s. for $b>0$ 
and every $i \geq 1$. 
Define $S_n=\sum\limits_{i=1}^n (X_i-E[X_i])$ and 
$v_n= \sum\limits_{i=1}^n E[X_i^2]$. 
Then, for every $t>0$, it holds that 
$P(S_n \geq t ) 
\leq 
\exp \Bigl\{
-\frac{t^2}{2v_n + \frac{2}{3}bt}
\Bigr\}$. 
By using this inequality with Lemma \ref{a1}, we obtain 
\begin{align*}
\mathbb{P}\bigl( 
|\{x\in \Lambda_{N}^L(z) ; 
\widetilde{m}_{N, L}^\eta(x) \geq 
(b \log N)^{\frac{1}{\alpha}}\}| 
\geq \frac{1}{|\widetilde{\Lambda}_{L}|} N^{d-Kb+\vare} \bigr)
\leq \exp\{ -N^{d-Kb+\frac{1}{2}\vare}\}, 
\end{align*}
for every $N$ large enough. 
The right hand side is summable in $N\geq 1$ and 
we obtain (\ref{151}) by Borel-Cantelli's lemma. 

Next, by the union bound and Lemma \ref{a1}, 
\begin{align*}
\mathbb{P}\bigl(\max_{x\in \Lambda_N}\widetilde{m}_{N, L}^\eta(x) 
\geq \bigl(\frac{d+1+\vare}{K}\log N\bigr)^{\frac{1}{\alpha}} \bigr) 
& \leq 
|\Lambda_N| \max_{x\in \Lambda_N} 
\mathbb{P}\bigl(\widetilde{m}_{N, L}^\eta(x) \geq 
\bigl(\frac{d+1+\vare}{K}\log N\bigr)^{\frac{1}{\alpha}}\bigr) \\
& \leq 
N^{-1-\frac{1}{2}\vare}, 
\end{align*}
for every $N$ large enough and this yields (\ref{152}). 
Finally, (ii) follows from Lemma \ref{a1} (iii) 
and Proposition \ref{gfunc}. 
\end{proof}

\smallskip

\begin{proof}[Proof of $(\ref{maxgub})$ and $(\ref{maxgub2})$ 
for the case $(A_\alpha)$\ $\alpha =2$]
Let $b>\sqrt{(2G^*+\frac{1}{K})d}$ 
where $K= \frac{c_2}{G_{(2)}^*}$. 
We consider the same discretization of the mean height 
as the proof of Proposition 2.1 of \cite{bib3}. 
For $\vare>0$ small enough and 
$\overline{j}\in \mathbb{N}$ large enough, 
set $\theta = \sqrt{\frac{d}{K}}\frac{1}
{\overline{j}}$ and 
$\widetilde{j}= \frac{1}{\theta}\sqrt{\frac{d+1+\vare}{K}}$. 
For given $m_N^\eta(x)$, define 
\begin{equation}\label{disc}
\overline{m}(x)= 
\begin{cases}
\theta \sqrt{\log N} & \text{ if } m_N^\eta(x) 
\in (-\infty, \theta \sqrt{\log N}], \\
j \theta \sqrt{\log N} & \text{ if } m_N^\eta(x) 
\in ( (j-1)\theta \sqrt{\log N}, j \theta \sqrt{\log N}],\ 
j =2, 3, \cdots, \overline{j}, \\
\widetilde{j} \theta \sqrt{\log N} & \text{ if } m_N^\eta(x) 
\in (\overline{j}\theta \sqrt{\log N}, 
\widetilde{j} \theta \sqrt{\log N}], \\ 
\infty & \text{ if } m_N^\eta(x) 
\in (\widetilde{j} \theta \sqrt{\log N}, \infty). 
\end{cases}
\end{equation}
We also define random sets $\mathcal{D}_N(j) = 
\{x \in \Lambda_{N}; \overline{m}(x) = 
j \theta \sqrt{\log N}\}$, $j\in \{1, 2, \cdots, \overline{j}, 
\widetilde{j}, \infty\}$ and consider the event: 
\begin{align*}
& \mathcal{A}_N = \bigl\{
|\mathcal{D}_N(j)| \leq N^{d-K(j-1)^2\theta^2+\vare} 
\text{ for every } j \in \{2, 3, \cdots, \overline{j}\}, \\
& \qquad \qquad \qquad \qquad \qquad \qquad \qquad 
|\mathcal{D}_N(\widetilde{j})| \leq N^{\vare} 
\text{ and } |\mathcal{D}_N(\infty)| =0\bigr\}. 
\end{align*}
Then, $\mathbb{P}(\mathcal{A}_N, 
\text{ for all large } N )=1$ by Lemma \ref{mht} 
and trivially $|\mathcal{D}_N(1)| \leq |\Lambda_{N}|$. 
We fix such a realization $\eta$. 
By FKG inequality, 
\begin{align*}
\mu_N^{\eta}\bigl(
\max_{x\in \Lambda_N}\phi(x) \leq b \sqrt{\log N} \bigr) 
& = 
\mu_N
\bigl(
\max_{x\in \Lambda_N}\{\phi(x) + m_N^\eta(x)\} 
\leq b \sqrt{\log N} \bigr) \\
& \geq 
\prod_{x\in \Lambda_N} \mu_N \bigl(\phi(x) \leq 
b \sqrt{\log N}-m_N^\eta(x)\bigr) \\
& \geq 
\exp\Bigl\{ \sum\limits_{x\in \Lambda_N} \log 
\bigl(1-
\mu_N\bigl( \phi(x) \geq 
b \sqrt{\log N}-\overline{m}(x)\bigr) \bigr)\Bigr\}. 
\end{align*}
Firstly, we prove (\ref{maxgub2}) under the condition 
$2dG^* > \frac{1}{K}$.
In this case, $b > \widetilde{j}\theta$ for 
$\vare>0$ small enough and we have 
\begin{align*}
T_N & := 
\sum\limits_{x\in \Lambda_N} \log 
\bigl(1-\mu_N\bigl(\phi(x) \geq 
b \sqrt{\log N}-\overline{m}(x)\bigr)\bigr) \\
& \geq 
\sum\limits_{j \in \{1, 2, \cdots, \overline{j}, \widetilde{j}\}}
|\mathcal{D}_N(j)| \log 
\bigl(1-\exp\bigl\{-\frac{1}{2G^*}(b -j \theta)^2
 \log N\bigr\}\bigr) \\
& \geq 
-2 \sum\limits_{j \in \{1, 2, \cdots, \overline{j}, \widetilde{j}\}}
|\mathcal{D}_N(j)| 
N^{-\frac{1}{2G^*}(b-j\theta)^2}, 
\end{align*}
for every $N$ large enough where we used Gaussian tail estimate 
and $\log(1-x) \geq -2x$ for every $x>0$ small enough. 
By definition of $\mathcal{A}_N$, in order to prove that the 
right hand side goes to 0 as $N\to \infty$ it 
is sufficient to show that 
\begin{equation*}
\max\limits_{j \in \{1, 2, \cdots, \overline{j}\}}
\bigl\{ d-K(j-1)^2 \theta^2 
-\frac{1}{2G^*}(b-j\theta)^2\bigr\}<0, 
\end{equation*}
and this inequality holds once we have 
\begin{equation}\label{601}
\min\limits_{0 \leq x \leq \sqrt{\frac{d}{K}}} 
\bigl\{ K x^2 + \frac{1}{2G^*}(b -x -\theta)^2\bigr\}>d. 
\end{equation}
The left hand side of (\ref{601}) equals to 
$\frac{K(b-\theta)^2}{2KG^*+1}$. 
Since $b>\sqrt{(2G^*+\frac{1}{K})d}$, if we take 
$\overline{j}$ large enough, then (\ref{601}) holds 
and we complete the proof of (\ref{maxgub2}). 

Next, we prove (\ref{maxgub}) for the general case. 
We take $\widetilde{j}$ as 
$\widetilde{j}= \frac{1}{\theta}\sqrt{\frac{d+\vare}{K}}$ 
in the discretization of $m_N^\eta$ (\ref{disc}). 
If $b >\sqrt{(2G^*+\frac{1}{K})d}$ then we can 
take $\vare>0$ small enough so that $b >\widetilde{j}\theta$. 
By Lemma \ref{a1}, 
\begin{align*}
\mathbb{P}(\max_{x\in \Lambda_N}m_N^\eta(x) \geq 
\widetilde{j}\theta \sqrt{\log N}) 
& \leq 
|\Lambda_N| \max_{x\in \Lambda_N} 
\mathbb{P} (m_N^\eta(x) \geq 
\sqrt{\frac{d+\vare}{K}}\sqrt{\log N}) \to  0,
\end{align*}
as $N\to \infty$. 
Therefore, 
we have $\lim\limits_{N\to \infty}\mathbb{P}(\mathcal{B}_N)=1$ 
where $\mathcal{B}_N$ is defined by the following: 
\begin{align*}
& \mathcal{B}_N = \bigl\{
|\mathcal{D}_N(j)| \leq N^{d-K(j-1)^2\theta^2+\vare} 
\text{ for every } j \in \{2, 3, \cdots, \overline{j}\} \\
& \qquad \qquad \qquad \qquad 
\text{ and }
|\mathcal{D}_N(\widetilde{j})| \leq N^{\vare}, 
\text{ for all large } N 
\bigr\} 
\cap 
\bigl\{ \max\limits_{x\in \Lambda_N}m_N^\eta(x) \leq 
\widetilde{j}\theta \sqrt{\log N}\bigr\}. 
\end{align*}
On $\mathcal{B}_N$ we can proceed the same argument as above and 
we have 
$
\mu_N^{\eta}\bigl(
\max\limits_{x\in \Lambda_N}\phi(x) \leq b \sqrt{\log N} \bigr) 
\geq 1-\delta_N$ for some $\delta_N \to 0$. 
Hence we obtain (\ref{maxgub}). 
\end{proof}

\smallskip

\begin{proof}[Proof of $(\ref{maxgub})$ for the case 
$(A_\alpha)$\ $\alpha \in (0, 2)$]
Assume the condition $(A_\alpha)$\ $\alpha \in (0, 2)$ and 
take $b >\frac{d}{K}$. 
In this case the maximum of $\phi$-field under 
$\mu_N^{\eta}$ is dominated by the maximum of $m_N^\eta$-field. 
By the union bound and Lemma \ref{a1}, 
$
\mathbb{P}\bigl(\max\limits_{x\in \Lambda_N}m_N^\eta(x) \geq 
\bigl(\frac{d+\vare}{K}\log N\bigr)^{\frac{1}{\alpha}} \bigr) 
\to  0$
as $N\to \infty$ for every $\vare>0$. 
We take $\vare>0$ small enough so that 
$b >\frac{d+\vare}{K}$. 
Then, on the event 
$\{ \max\limits_{x\in \Lambda_N}m_N^\eta(x) \leq 
\bigl(\frac{d+\vare}{K}\log N\bigr)^{\frac{1}{\alpha}} \}$, 
we have 
\begin{align*}
\mu_N^{\eta}
\bigl(
\max_{x\in \Lambda_N}\phi(x) \leq (b {\log N})^{\frac{1}{\alpha}} \bigr) 
& \geq 
\prod_{x\in \Lambda_N} \mu_N \bigl(\phi(x) + m_N^\eta(x) \leq 
(b{\log N})^{\frac{1}{\alpha}} \bigr)\\
& \geq 
\Bigl(1- 
\exp\bigl\{-\frac{1}{2G^*}\bigl( 
(b {\log N})^{\frac{1}{\alpha}} - 
\bigl(\frac{d+\vare}{K}\log N\bigr)^{\frac{1}{\alpha}}\bigr)^2\bigr\}
\Bigr)^{CN^d}\\
& \to 1, 
\end{align*}
as $N\to \infty$. 
\end{proof}

\smallskip

\begin{proof}[Proof of $(\ref{maxgub2})$ for the case 
$(\widetilde{A})$]
By Lemma \ref{mht} (ii) we fix a realization 
$\eta \in \{ \max\limits_{x\in \Lambda_N}m^\eta_N(x) \leq 
\vare \sqrt{\log N}, \text{ for all large } N \}$. 
Then, $\max\limits_{x\in \Lambda_N}m^\eta_N(x)$ 
does not affect the maximum of the $\phi$-field and 
 (\ref{maxgub2}) easily follows from FKG inequality and Gaussian tail 
estimate. 
\end{proof}
\begin{remark}\label{rem1}
As seen in this proof, under the condition $(A_\alpha)$ 
the typical maximum of the mean height 
is given by $\max\limits_{x\in \Lambda_N} m_N^\eta(x) 
\approx ({\frac{d}{K}\log N})^{\frac{1}{\alpha}}$. 
On the other hand, as a $\mathbb{P}$-almost sure result 
we are only able to prove that 
$\mathbb{P}\bigl( 
\max\limits_{x\in \Lambda_N} m_N^\eta(x) 
\leq (\frac{d+1+\vare}{K}\log N)^{\frac{1}{\alpha}}, 
\text{ for all large } N \bigr)=1$ 
by using Borel-Cantelli's lemma. 
The probability of the event $\bigl\{  
\max\limits_{x\in \Lambda_N}m_N^\eta(x) \in 
\bigl[(\frac{d+\vare}{K}\log N)^{\frac{1}{\alpha}}, 
(\frac{d+1+\vare}{K}\log N)^{\frac{1}{\alpha}}
\bigr]\bigr\}$ goes to $0$ as $N\to \infty$ but 
this is not negligible and the maximum of the $\phi$-field 
can be large on this event. 
For this reason the upper bound of the maximum 
is given in the form of $(\ref{maxgub})$ 
when $\alpha \in (0, 2)$ or, $\alpha=2$ and 
$2dG^* \leq \frac{G_{(2)}^*}{c_2}$. 
\end{remark}


\renewcommand{\thesection}{\Alph{section}}
\setcounter{section}{0}
\section{Appendix}
\subsection{The infinite volume case}
We consider $\max\limits_{x\in \Lambda_N}\phi(x)$ 
under the infinite volume Gaussian Gibbs measure 
with disorder. 
Recall that the Gibbs measure $\mu_N^\eta$ given by 
(\ref{gibbs}) coincides with the law of the 
Gaussian field on $\mathbb{R}^{\Lambda_N}$ 
with the mean 
$m_N^\eta(x)= \sum\limits_{y \in \Lambda_N}G_N(x, y)\eta(y)$, 
$x\in \Lambda_N$  and the covariance matrix 
$(-\Delta_N)^{-1}(x, y)$, 
$x, y\in \Lambda_N$. 
Let $\mu_\infty^\eta$ be 
the law of the 
Gaussian field on $\mathbb{R}^{\mathbb{Z}^d}$ 
with the mean 
$m_\infty^\eta(x)= \sum\limits_{y \in \mathbb{Z}^d}
G(x, y)\eta(y)$, 
$x\in \mathbb{Z}^d$  and the covariance matrix 
$(-\Delta)^{-1}(x, y)$, $x, y\in \mathbb{Z}^d$. 
We first show that 
$\mu_\infty^\eta$ exists and 
$\mu_N^\eta$ weakly converges to $\mu_\infty^\eta$ 
when $d\geq 5$. 
\begin{prop}\label{conv}
Let $d\geq 5$ and 
$\eta=\{\eta(x)\}_{x\in \mathbb{Z}^d}$ be a family of i.i.d. 
random variables with mean $0$ and variance 
$\sigma^2 \in (0, \infty)$. 
Then, for $\mathbb{P}$-almost every $\eta$, 
$\mu_N^\eta$ weakly converges to $\mu_\infty^\eta$ 
as $N\to \infty$. 
\end{prop}
\begin{proof}
First of all, we remark that 
\begin{align*}
\mathbb{E}[(m_\infty^\eta(x))^2] 
& = \sigma^2 \sum\limits_{y \in \mathbb{Z}^d}G(x, y)^2 
=\sigma^2 (-\Delta)^{-2}(0, 0) <\infty,
\end{align*}
for every $x\in \mathbb{Z}^d$ when $d\geq 5$. 
Hence, for $\mathbb{P}$-almost every realization 
$\eta$, 
$|m_\infty^\eta(x)|<\infty$ for every $x\in \mathbb{Z}^d$ 
and the Gaussian measure $\mu_\infty^\eta$ exists 
when $d\geq 5$. 
Then, it is sufficient to show the convergence 
of its mean and variance
since $\mu_N^\eta$ is a Gaussian measure. 
The convergence of the variance 
$(-\Delta_N)^{-1}(x, y) \to 
(-\Delta)^{-1}(x, y),\, x, y\in \mathbb{Z}^d$ 
holds when $d\geq 3$. 
For the convergence of the mean 
let $x\in \mathbb{Z}^d$ be fixed 
and we compute that
\begin{align*}
\mathbb{E}\bigl[
(m_\infty^\eta(x)-m_N^\eta(x))^2
\bigr] 
& = 
\mathbb{E}\Bigl[
\Bigl\{ \sum\limits_{y \in \Lambda_N^c}G(x, y)\eta(y) + 
\sum\limits_{y \in \Lambda_N}
\bigl(G(x, y)-G_N(x, y)\bigr) \eta(y)\Bigr\}^2\Bigr] \\
& \leq 
2\mathbb{E}\Bigl[
\Bigl\{ \sum\limits_{y \in \Lambda_N^c}G(x, y)\eta(y)\Bigr\}^2 
\Bigr] + 
2\mathbb{E}\Bigl[
\Bigl\{\sum\limits_{y \in \Lambda_N}
\bigl(G(x, y)-G_N(x, y)\bigr) \eta(y)\Bigr\}^2 \Bigr] \\
& =: I_N^1 + I_N^2, 
\end{align*}
where the inequality follows from 
$(a+b)^2 \leq 2a^2+ 2b^2$ for every $a, b\in \mathbb{R}$. 
By the assumption on $\eta$ and Proposition \ref{gfunc}, 
\begin{align*}
I_N^1 & 
= 2 \sigma^2 \sum\limits_{y \in \Lambda_N^c} G(x, y)^2 
\leq C\sum\limits_{r \geq N} r^{d-1}(\frac{1}{r^{d-2}})^2 
\leq C N^{-d+4}. 
\end{align*}
For the estimate on $I_N^2$ we 
show that there exists a constant $C>0$ such that 
$0\leq G(x, y)-G_N(x, y)\leq CN^{-d+2}$ 
for every $y\in \Lambda_N$ and every $N$ large enough. 
Take $\vare \in (0, 1)$. 
For $y\in \Lambda_N \cap \Lambda_{\vare N}^c$, 
we have a crude bound: 
\begin{align*}
0 \leq G(x, y)-G_N(x, y) \leq 
G(x, y) \leq \frac{C}{|x-y|^{d-2}} \leq CN^{-d+2}. 
\end{align*}
For $y\in \Lambda_{\vare N}$, 
\begin{align*}
0 \leq G(x, y)-G_N(x, y) = 
\sum\limits_{z \in \partial^+\Lambda_N}
P_x(S_{\tau_{\Lambda_N}} =z)G(z, y) 
\leq CN^{-d+2}, 
\end{align*}
where the equality follows from 
\cite[Proposition 1.5.8]{L} and the last 
inequality follows from 
$G(z, y) \leq \frac{C}{|z-y|^{d-2}} \leq 
CN^{-d+2}$ for every 
$y\in \Lambda_{\vare N}$ and 
$z\in \partial^+\Lambda_N$. 
Therefore, we have 
\begin{align*}
I_N^2 & 
= 2 \sigma^2 
\sum\limits_{y \in \Lambda_N}
\bigl(G(x, y)-G_N(x, y)\bigr)^2
\leq CN^d({N^{-d+2}})^2 
\leq C N^{-d+4}. 
\end{align*}
By collecting all the estimates, 
$m_N^\eta(x)$ converges to 
$m_\infty^\eta(x)$ in $L^2$ and hence in probability. 
Then, by L\'evy's theorem (cf. \cite[Theorem 3.9]{Var}) 
$m_N^\eta(x)$ converges to 
$m_\infty^\eta(x)$ $\mathbb{P}$-a.s. 
and this completes the proof. 
\end{proof}
\begin{remark}
$\mu_N^\eta$ does not converge when $d\leq 4$. 
Non-existence of infinite volume 
Gaussian Gibbs measures in random external fields 
has been discussed in 
\cite[Appendix A.1]{CK12}. 
\end{remark}
By Proposition \ref{conv} and a standard argument 
(cf. \cite[Proposition 3.4]{F05} for example) 
it is easy to see 
that $\mu_\infty^\eta$ is an infinite volume 
(random) $\phi$-Gibbs measure, namely 
for $\mathbb{P}$-almost every realization $\eta$, 
$\mu_\infty^\eta$ satisfies the DLR equation: 
$$
\int_{\mathbb{R}^{\mathbb{Z}^d}} 
f(\phi) \mu_\infty^\eta(d\phi)
= 
\int_{\mathbb{R}^{\mathbb{Z}^d}}  \Bigl(\int_{\mathbb{R}^{\Lambda}} 
f(\phi) \mu_\Lambda^{\eta, \psi}(d\phi)
\Bigr)
\mu_\infty^\eta(d\psi), 
$$
for every finite $\Lambda \subset \mathbb{Z}^d$ 
and for all $f\in C_b(\mathbb{R}^{\mathbb{Z}^d})$ 
where $\mu_\Lambda^{\eta, \psi}$ denotes the 
finite volume Gibbs measure on $\mathbb{R}^\Lambda$ 
with $\psi$-boundary condition outside $\Lambda$ 
defined by the following. 
\begin{align*}
\mu_\Lambda^{\eta, \psi}(d\phi) 
=\frac{1}{Z_\Lambda^{\eta, \psi}}
\exp\Bigl\{
-\frac{1}{4d}\sum_{\substack{\{x, y\}\cap \Lambda\ne\phi \\ |x-y|=1}}
(\phi(x)-\phi(y))^2
 + \sum\limits_{x\in \Lambda}\eta(x)\phi(x) \Bigr\}
\prod\limits_{x\in \Lambda}d\phi(x)
\prod\limits_{x\in \mathbb{Z}^d\setminus \Lambda}
\delta_{\psi(x)}(d\phi(x)). 
\end{align*}

Then, the completely same statement to 
Theorem \ref{thm1} holds for $\max\limits_{x\in \Lambda_N}\phi(x)$ 
under the measure $\mu_\infty^\eta$. 
The reason is that we can show the similar estimates 
as Lemma \ref{a1} for 
the infinite weighted sum 
$m_\infty^\eta(x)= \sum\limits_{y \in \mathbb{Z}^d}
G(x, y)\eta(y)$, $x\in\mathbb{Z}^d$ 
by using the fact that 
$\sum\limits_{y \in \Lambda_N^c}G(x, y)\eta(y)$ goes 
to $0$ as $N\to \infty$ for every $x\in \mathbb{Z}^d$, $\mathbb{P}$-a.s. 
which follows from the proof of Proposition \ref{conv}. 
The argument in Section 3 also works well 
for $\mu_\infty^\eta$
by using the DLR equation. 
All the modifications are straightforward.

\subsection{The lower dimensional case}
In the lower dimensional case 
$1\leq d\leq 4$, 
fluctuation of the mean height $m_N^\eta$ is larger than 
that of the original $\phi$-field and 
$m_N^\eta$-field has long range correlations. 
By these reasons the quenched result such as Theorem \ref{thm1} 
is not expected. 
We have the following example. 
Let $1\leq d \leq 3$. 
To avoid the constraint by $0$-boundary conditions, 
we consider $\max\limits_{x\in \Lambda_{(1-\vare)N}}\phi(x)$, 
$\vare\in (0, 1)$ 
instead of $\max\limits_{x\in \Lambda_{N}}\phi(x)$. 
We take 
$\eta = \{\eta(x)\}_{x\in \mathbb{Z}^d}$ as a family of i.i.d. 
standard normal random variables. 
In this case 
$m_N^{\eta}= \{ m_N^{\eta}(x)\}_{x\in \Lambda_N}$ 
is a centered Gaussian field on $\mathbb{R}^{\Lambda_N}$ whose 
covariance matrix is given by $(-\Delta_N)^{-2}$. 
This is so-called Gaussian membrane model and 
its scaling limit has been studied in 
\cite{bib7} and \cite{CDH19}. 
Let $\Psi$ be a scaled field define by 
$\Psi_N(t)=\frac{1}{N^{2-\frac{d}{2}}}m_N^\eta(Nt)$ for 
$t \in \frac{1}{N}\mathbb{Z}^d$ combined with some 
continuous interpolation. 
Then, the field $\{ \Psi_N(t) \}_{t \in [-1, 1]^d}$ converges 
in distribution to some centered continuous Gaussian process 
on $\mathbb{R}^{[-1, 1]^d}$ in the space of all 
continuous functions on $[-1, 1]^d$ 
(cf. \cite{CDH19} for $d=2, 3$ and \cite{bib7} for $d=1$). 
In particular, if we consider the events: 
\begin{align*}
\mathcal{A}_N^+& = \{m_N^{\eta}(x) \geq  \delta N^{2-\frac{d}{2}} 
\text{ for every } x\in \Lambda_{(1-\vare)N}\}, \\
\mathcal{A}_N^-& = \{m_N^{\eta}(x) \leq -\delta N^{2-\frac{d}{2}} 
\text{ for every } x\in \Lambda_{(1-\vare)N}\}, 
\end{align*}
for $\delta>0$, then there exist some constants $q_+, q_->0$ such that 
$\mathbb{P}(\mathcal{A}_N^+)\geq q_+$, 
$\mathbb{P}(\mathcal{A}_N^-)\geq q_-$ for every $N$ large 
enough.  

On the event $\mathcal{A}_N^+$, we have 
\begin{align*}
\mu_N^{\eta}\bigl( \max\limits_{x\in \Lambda_{(1-\vare)N}} 
\phi(x) \geq \frac{1}{2} \delta N^{2-\frac{d}{2}} \bigr) 
& \geq 
\mu_N^{\eta}\bigl(\phi(0) - m_N^{\eta}(0) \geq 
-\frac{1}{2} \delta N^{2-\frac{d}{2}} \bigr) \\
& \geq 
1-e^{-C\frac{\delta^2 N^{4-d}}{G_N(0, 0)}}
\to 1, 
\end{align*}
as $N\to \infty$. 
On the other hand, on the event $\mathcal{A}_N^-$, 
\begin{align*}
\mu_N^{\eta}\bigl( \max\limits_{x\in \Lambda_{(1-\vare)N}} 
\phi(x) \geq -\frac{1}{2}\delta N^{2-\frac{d}{2}} \bigr) 
& \leq 
\mu_N^{\eta}\bigl(\phi(x) - m_N^{\eta}(x) \geq 
\frac{1}{2}\delta N^{2-\frac{d}{2}} \text{ for some }x \in 
\Lambda_{(1-\vare)N} \bigr) \\
& \leq 
CN^d e^{-C\frac{\delta^2 N^{4-d}}{G_N(0, 0)}}
\to 0. 
\end{align*}
As a result, we have 
\begin{align*}
& \mathbb{P}\Bigl(
\lim\limits_{N\to \infty} 
\mu_N^{\eta}\bigl( \max\limits_{x\in \Lambda_{(1-\vare)N}} 
\phi(x) \geq \delta N^{2-\frac{d}{2}} \bigr) =1 
\Bigr) \geq q_+, 
\end{align*}
and 
\begin{align*}
& \mathbb{P}\Bigl(
\lim\limits_{N\to \infty} 
\mu_N^{\eta}\bigl( \max\limits_{x\in \Lambda_{(1-\vare)N}} 
\phi(x) \leq -\delta N^{2-\frac{d}{2}} \bigr) =1 
\Bigr) \geq q_-. 
\end{align*}
\subsection{Entropic repulsion}
As explained in the previous sections, 
the proof of Theorem \ref{thm1} is closely related to 
the problem of entropic repulsion. 
Consider the hard wall event 
$\Omega_{N, \vare}^+ =\bigl\{\phi \, ; \phi(x) \geq 0 
\text{ for every } x\in \Lambda_{(1-\vare)N} \bigr\}$, $\vare\in (0, 1)$. 
Under the condition that the law of $\eta$ is symmetric, we have 
\begin{align*}
\mu_N^{\eta}\bigl( \Omega_{N, \vare}^+ \bigr) 
& = 
\mu_N\bigl( \phi(x)+m_N^\eta(x) \geq 0 \text{ for every } 
x \in \Lambda_{(1-\vare)N} \bigr) \\
& \stackrel{d}{=} 
\mu_N\bigl( \phi(x) \geq m_N^\eta(x) \text{ for every } 
x \in \Lambda_{(1-\vare)N} \bigr), 
\end{align*}
where $\stackrel{d}{=}$ 
means equivalence in law as a $\sigma(\eta(x); x\in \mathbb{Z}^d)$ 
-measurable random variable. 
Therefore, the asymptotic behavior of 
$\mu_N^{\eta}\bigl( \Omega_{N, \vare}^+ \bigr)$ 
corresponds to the problem of entropic repulsion 
on the random wall $m_N^\eta=\{m_N^\eta(x)\}_{x\in \Lambda_N}$. 
If we take $L$ large enough, then by Lemma \ref{a1} (i)
the field of the mean height 
$\{m_N^\eta(y); y \in \Gamma_{N, \vare}\}$ 
can be treated as independent random variables where 
$\Gamma_{N, \vare}:= \Lambda_{(1-\vare)N}\cap 4L\mathbb{Z}^d$. 
Additionally, by Lemma \ref{a1} (ii) we know its precise tail asymptotics. 
By these facts we can proceed the similar argument to the proof of 
\cite{bib3} which studied the problem of entropic repulsion for 
the i.i.d. random wall case. 
(Actually, \cite{bib3} mentioned the case of the weakly correlated wall.) 
All the modifications are almost straightforward (but lengthy and messy). 
We introduce the result without the proof. 
\begin{prop}\label{thm2}
Let $d\geq 5$, $\vare \in (0, 1)$ 
and $\eta=\{\eta(x)\}_{x\in \mathbb{Z}^d}$ be a family of i.i.d. 
symmetric random variables. 
Assume the condition $(A_\alpha)$ or $(\widetilde{A})$.  
Then, the following holds for $\mathbb{P}$-$\mathrm{a.s.}$: 
\begin{equation*}
\lim_{N\to\infty}
\frac{1}{N^{d-2}(\log N)^{\frac{2}{\alpha \wedge 2}}}\log \mu_N^{\eta}
 \bigl( \Omega_{N, \vare}^+ \bigr) 
= - R^* \mathrm{Cap}((-1+\vare, 1-\vare)^d), 
\end{equation*}
where 
$\mathrm{Cap}((-1+\vare, 1-\vare)^d)$ is the Newtonian capacity of 
$(-1+\vare, 1-\vare)^d \subset \mathbb{R}^d$ defined by 
\begin{equation*}
\mathrm{Cap}((-1+\vare, 1-\vare)^d) = 
\inf\bigl\{ \frac{1}{2d}\| \partial f\|_{L^2}^2 ; 
f\in C_0^\infty(\mathbb{R}^d),\, f\geq 0,\, f\equiv 1 \ \mathrm{ on }\ 
(-1+\vare, 1-\vare)^d\bigr\}, 
\end{equation*}
and $R^*$ is given by the following: 
\begin{equation*}
R^* = 
\begin{cases}
\frac{1}{2}(\frac{2}{c_\alpha})^{\frac{2}{\alpha}} (G^*)^{2} 
& \text{ if } (A_\alpha)\ \alpha \in (0, 1], \\
\frac{1}{2}(\frac{2}{c_\alpha})^{\frac{2}{\alpha}} 
(G^*_{(\alpha)})^{\frac{2\alpha-2}{\alpha}} 
& \text{ if } (A_\alpha)\ \alpha \in (1, 2), \\
2G^* + \frac{{G}_{(2)}^*}{c_2}
& \text{ if } (A_\alpha)\ \alpha =2, \\
2G^* & \text{ if } (\widetilde{A}). 
\end{cases}
\end{equation*}
\end{prop}

\section*{Acknowledgement}
This work was partially supported 
by JSPS KAKENHI Grant Number 22K03359. 

\end{document}